\documentclass[reqno,a4paper,12pt]{amsart}
\usepackage{amssymb,amsmath,amscd,amstext,amsthm,amsfonts}
\usepackage[]{graphicx}
\usepackage{epstopdf}
\usepackage{setspace} 
\usepackage{subfig}
\usepackage{color}
\usepackage[mathscr]{eucal}
\usepackage[ansinew]{inputenc} 
\usepackage{hyperref}
\usepackage[shortlabels]{enumitem}

\numberwithin{equation}{section}

\newtheorem{theorem}{Theorem}[section]
\newtheorem{proposition}[theorem]{Proposition}
\newtheorem{conjecture}[theorem]{Conjecture}
\newtheorem{corollary}[theorem]{Corollary}
\newtheorem{lemma}[theorem]{Lemma}

{\theoremstyle{definition}
{
\newtheorem{remark}[theorem]{Remark}

\newtheorem{defn}[theorem]{Definition}}}

\newcommand{\cal}{\mathcal}

\newcommand{\BB}{{\cal B}}

\newcommand{\FF}{{\cal F}}
\newcommand{\GG}{{\cal G}}

\newcommand{\PP}{{\cal P}}

\newcommand{\RR}{{\cal R}}

\newcommand{\XX}{{\cal X}}

\newcommand{\SRB}{\mathscr{E}}

\newcommand{\Nn}{{\mathbb{N}}}

\newcommand{\Pp}{{\mathbb{P}}}

\newcommand{\Rr}{{\mathbb{R}}}

\def\supp{\operatorname{supp}}

\def\Vol{\operatorname{Vol}}

\newcommand{\comment}[1]{}

\begin{document}
\title[Hyperbolic polygonal billiard]{Hyperbolic polygonal billiards close to 1-dimensional piecewise expanding maps}

\date{\today}

\author[Del Magno]{Gianluigi Del Magno}
\address{Universidade Federal da Bahia, Instituto de Matem\'atica\\
Avenida Adhemar de Barros, Ondina \\
40.170.110 - Salvador, BA - Brasil}
\email{gdelmagno@ufba.br}

\author[Lopes Dias]{Jo\~ao Lopes Dias}
\address{Departamento de Matem\'atica, CEMAPRE and REM, ISEG\\
Universidade de Lisboa\\
Rua do Quelhas 6, 1200-781 Lisboa, Portugal}
\email{jldias@iseg.ulisboa.pt}

\author[Duarte]{Pedro Duarte}
\address{Departamento de Matem\'atica and CMAF \\
Faculdade de Ci\^encias\\
Universidade de Lisboa\\
Campo Grande, Edificio C6, Piso 2\\
1749-016 Lisboa, Portugal 
}
\email{pmduarte@fc.ul.pt}

\author[Gaiv\~ao]{Jos\'e Pedro Gaiv\~ao}
\address{Departamento de Matem\'atica, CEMAPRE and REM, ISEG\\
Universidade de Lisboa\\
Rua do Quelhas 6, 1200-781 Lisboa, Portugal}
\email{jpgaivao@iseg.ulisboa.pt}

\begin{abstract}
We consider polygonal billiards with collisions contracting the reflection angle towards the normal to the boundary of the table. In previous work, we proved that such billiards has a finite number of ergodic SRB measures supported on hyperbolic generalized attractors. Here we study the relation of these measures with the ergodic absolutely continuous invariant probability measures (acips) of the slap map, the 1-dimensional map obtained from the billiard map when the angle of reflection is always equal to zero. Our main result states that for a generic polygon, if the reflection law has a Lipschitz constant sufficiently small, then there exists a one-to-one correspondence between the ergodic SRB measures of the billiard map and the ergodic acips of the corresponding slap map, and moreover that the number of Bernoulli components of each ergodic SRB measure equals the number of the exact components of the corresponding ergodic acip. The case of billiards in regular polygons and triangles is studied in detail.
\end{abstract}

\maketitle
\tableofcontents

\nocite{*}

\section{Introduction}
\label{sec:introduction}
The dynamics of billiards has been studied in great detail when the reflection law is the specular one: the angle of reflection equals the angle of incidence. For an account on the subject, we refer the reader to~\cite{G96,S00,tabachnikov95}. In a series of recent works, we studied polygonal billiards with a reflection law, i.e., a function $ f $ describing the dependence of the angle of reflection from the angle of incidence -- both measured with respect to the normal of the billiard table -- that is not the identity function as for the specular reflection law, but a strict contraction having the zero angle as its fixed point (see Fig.~\ref{fig:reflection law}). 

The dynamics of polygonal billiards with contracting reflection laws differ significantly from that of polygonal billiards with specular reflection law: whereas the latter are non-hyperbolic systems, the former generically have uniformly hyperbolic attractors supporting a finite number of ergodic Sinai-Ruelle-Bowen measures (SRB measures for short)~\cite{MDDGP12,MDDGP13,lxds}. Some billiards in non-polygonal tables with non-specular reflection laws were studied in~\cite{arroyo09,arroyo12,markarian10}. 

In the degenerate case when the function $ f $ is identically equal to 0, i.e., when the angle of reflection $ \theta $ is identically equal to zero, the billiard map is no longer injective and its image is a 1-dimensional set. The restriction of the billiard map to this subset is a piecewise affine maps of the interval, which we call \emph{slap map} following~\cite{markarian10}. The precise form of the slap map depends only on the polygonal table of the billiard. If a polygon does not have parallel sides (in fact, a weaker condition introduced later on suffices), then the corresponding slap map is uniformly expanding, and admits a finite number of ergodic absolutely continuous invariant probabilities  (acips for short)~\cite{MDDGP13-1}.

Given a polygon $ P $ and a contracting reflection law $ f $, we denote by $ \Phi_{f,P} $ the map of the billiard in $ P $ with reflection law $ f $, and by $ \psi_P $ the slap map of $ P $. Precise definitions will be given in Section~\ref{se:polygonal}. The Lipschitz constant $ \lambda(f) $ of $ f $ may be thought as a measure of how the 2-dimensional map $ \Phi_{f,P} $ is close to the 1-dimensional map $ \psi_{P} $. More precisely, $ \lambda(f) $ measures how close $ \Phi_{f,P} $ is to $ \Phi_{0,P} $ (here $ f \equiv 0 $), but the image of $ \Phi_{0,P} $ is 1-dimensional and the restriction of $ \Phi_{0,P} $ to its image coincides with $ \psi_P $. In this paper, we address the natural question `what is the relation between the properties of $ \Phi_{f,Q} $ and $ \psi_{P} $ when $ \lambda(f) $ is small and the polygon $ Q $ is close to $ P $?'. In particular, we study the relation between the ergodic SRB measures of $ \Phi_{f,Q} $ and the ergodic acips of $ \psi_{P} $.

The results presented here are formulated for convex polygons only. Analogous results can be obtained for some classes of non-convex polygons, but the analysis would be much more involved than for convex polygons.
Our standing assumption is that a convex polygon $ P $ satisfies a condition called Condition~(*): every pair of billiard trajectories starting at a non-acute vertex $ v $ of $ P $, each moving in the direction perpendicular to one of the sides meeting at $ v $ do not visit any vertex of $ P $ and do not eventually become periodic (see~Section~\ref{sec:ergodic stability}). 

The main result of the paper is the following theorem. Its formulation and the subsequent statements in this section requires a distance $ d $ and a measure $ m $ on the space $ \PP_n $ of convex polygons with $ n $ sides. Their definitions are postponed to Section~\ref{sec:ergodic stability}. 

\begin{theorem}
\label{th:main}
Let $ P \in \PP_n $ be a polygon satisfying Condition~(*). There exists $ \delta>0 $ such that if $ f $ is a contracting reflection law that is a $ C^2 $ embedding with $ \lambda(f)<\delta $ and 
$ Q $ is a polygon of $ \PP_n $ with $ d(Q,P)<\delta $, then
there is a bijection between the set of the ergodic SRB measures of $ \Phi_{f,Q} $ and the set of the ergodic acips of $ \psi_P$. Moreover, 
\begin{enumerate}
\item the supports of the ergodic SRB measures of $ \Phi_{f,Q} $ are pairwise disjoint,
\item the number of the ergodic SRB measures of $ \Phi_{f,Q} $ is less than or equal to $ n $, 
\item the number of Bernoulli components of each ergodic SRB measure of $ \Phi_{f,Q} $ equals the number of exact components of the corresponding ergodic acip of $ \psi_P $,
\item the union of the basins of the ergodic SRB measures of $ \Phi_{f,Q} $ is a set of full volume in the domain of $ \Phi_{f,Q} $.
\end{enumerate}
\end{theorem} 

The previous result raises the natural question whether each ergodic SRB measures of $ \Phi_{f,Q} $ converges to the corresponding acip of $ \psi_P $ in the weak-* topology. We believe that the answer is positive, but we do not address that question in this paper. Important results on the existence and stability of physical measures of certain 2-dimensional hyperbolic maps with singularities were obtained by Demers and Liverani using a functional analytic approach in~\cite{DL}.  
Their results do not apply -- at least directly -- to the maps $ \Phi_{f,Q} $ in Theorem~\ref{th:main}, because both the unperturbed and perturbed maps are assume to be piecewise diffeomorphisms in~\cite{DL}, whereas it can be easily seen that the unperturbed map $ \Phi_{0,P} $ of $ \Phi_{f,Q} $ is not a piecewise diffeomorphism. Perhaps, it is possible to adapt the approach of~\cite{DL} to our setting following~\cite{D2018} where dissipative Baker's map and its limit map as the contraction coefficient converges to zero are considered. We have not tried to do that, and opted for a geometrical approach instead. 


Condition~(*) plays a major role in our analysis, and so it is important to know whether the set of polygons satisfying Condition~(*) is large in the topological and measure theoretical sense. In this regard, we proved the following (see Theorem~\ref{le:generic}).

\begin{theorem}
The set of polygons in $ \PP_n $ satisfying Condition~(*) is a residual set in $ \PP_n $ equipped with the distance $ d $ and full measure set in $ \PP_n $ equipped with the measure $m$.
\end{theorem}

Combining the previous theorems, one obtains that the set of polygons for which the conclusion of Theorem~\ref{th:main} holds is generic and has full measure in $ \PP_n $.

%




An obvious consequence of Condition~(*) is that $ \psi_P $ and every map $ \Phi_{f,Q} $ sufficiently close to $ \psi_P $ are hyperbolic. The general strategy of the proof of Theorem~\ref{th:main} is as follows. For every ergodic acip $ \nu $ of $ \psi_P $, we construct a trapping set $ W(\nu) $ common to all $ \Phi_{f,Q} $ sufficiently close to $ \psi_P $. We then consider one of the maps $ \Phi_{f,Q} $, and to establish the bijection between the ergodic acips of $ \psi_P $ and the ergodic SRB measures of $ \Phi_{f,Q} $, we prove that the support of each ergodic SRB measure of $ \Phi_{f,Q} $ is contained in a trapping set $ W(\nu) $ for some ergodic acip $ \nu $ of $ \psi_P $, and that for every ergodic acip $ \nu $ of $ \psi_P $, the corresponding trapping set $ W(\nu) $ contains the support of exactly one ergodic SRB measure of $ \Phi_{f,Q} $.
An interesting feature of the previous proof is that it is built upon properties of periodic points of the slap map $ \psi_P $ that can be transferred to periodic points of billiard maps $ \Phi_{f,Q} $ sufficiently close to $ \psi_P $. In Section~\ref{se:preliminary}, we collect several results that are needed to prove Theorem~\ref{th:main}. One of these results is a criterion for the ergodicity of SRB measures of $ \Phi_{f,Q} $ using the periodic points of $ \Phi_{f,Q} $. This criterion can be easily generalized to general hyperbolic map with singularities. A more detail description of the general structure of the proof Theorem~\ref{th:main} is provided in Section~\ref{sec:ergodic stability}.

We believe that our proof of Theorem~\ref{th:main} can be adapted to cover a quite general class of two-dimensional hyperbolic maps with singularities close in a proper sense to piecewise expanding 1-dimensional maps.

In the last section of the paper, we apply Theorem~\ref{th:main} to billiards in convex regular polygons and triangles, obtaining the following. 

\begin{corollary}
\label{co:regular}
Let $ P $ be a convex regular polygon with an odd number $ n $ sides. There exists $ \delta>0 $ such that if $ f $ is a contracting reflection law that is a $ C^2 $ embedding with $ \lambda(f)<\delta $ and $ Q \in \PP_n $ with $ d(Q,P)<\delta $, then the billiard map $ \Phi_{f,Q} $ has a unique ergodic SRB measure if $ n=3 $ or $ n=5 $, and exactly $ n $ ergodic SRB measures if $ n \ge 7 $.
\end{corollary}

%

\begin{corollary}
\label{co:acute}
For every acute triangle $P$, there exists $ \delta>0 $ such that if $ f $ is a contracting reflection law that is a $ C^2 $ embedding with $ \lambda(f)<\delta $, then the billiard map $ \Phi_{f,P} $ has a unique ergodic SRB measure. Moreover, this measure has a single Bernoulli component.
\end{corollary}


The paper is organized as follows. In Section~\ref{se:polygonal}, we define the objects studied in this paper: the billiard map for a polygonal billiard with a contracting reflection law and the related slap map. In Section~\ref{se:hyperbolic}, we give a sufficient condition for the existence of hyperbolic attractors of a billiard map, and recall the basic notions of Pesin's Theory specialized to our billiards. We also recall a general result on the  existence and the spectral decomposition of absolute invariant probabilities of piecewise expanding maps, which applies to the slap maps considered in this paper. Section~\ref{sec:ergodic stability} is devoted to the precise formulation of Condition~(*) and of the principal results of the paper. Besides giving a precise definition of the topology and measure in the space $\PP_n$, Section~\ref{sec:ergodic stability}  contains also an outline of the proof of Theorem~\ref{th:main} meant to guide the reader through it. The preliminary results necessary to prove the theorem are collected in Section~\ref{se:preliminary}, whereas the final part of its proof is contained in Section~\ref{sec:ergodic stability}. In Section~\ref{sec:examples}, we apply our general results to several classes of polygons, and prove Corollaries~\ref{co:regular} and~\ref{co:acute}.
\section{Billiard map and slap map}
\label{se:polygonal}

A billiard in a polygon $ P $ is a mechanical system formed by a point-particle moving with uniform motion inside $ P $ and bouncing off the boundary $ \partial P $ according to a given rule, which is a function called \emph{reflection law} whose argument and value are, respectively, the angle of incidence and the angle of reflection of the particle at the collision point. In the usual definition of a billiard, the reflection law is the specular one prescribing the equality between the angle of reflection and the angle of incidence. In this paper, we consider reflection laws that are strict contractions with small (in a sense that will be explained later) Lipschitz constant.

\subsection{Polygonal billiards}
\label{su:billiards} 
Let $ P $ be a convex polygon with $ n $ sides and perimeter equal to 1. We choose a positively oriented parametrization of $ \partial P $ by arc length so that $ 0 = s_{0} < s_1< \cdots < s_{n-1} < s_n = 1 $ are the values of the arc length parameter corresponding to the vertices of $ P $. The values $ s=0 $ and $ s=1 $ correspond to the same vertex of $ P $.
In the following, we identify the points of $ \partial P $ with their arc length parameter $ s $ (with the additional proviso that $ s=0 $ and $ s=1 $ denote the same point). In other words, we identify $ \partial P $ with the circle $ S^1 $ of perimeter 1. Denote by $ V_P = \{s_0,\ldots,s_{n-1}\} $ the set of the vertices of $ P $.

Let $ M = S^1 \times (-\pi/2,\pi/2) $. We denote by $ d_{S^1} $ the standard distance on $ S^1 $, and by $ d_{M} $ the Euclidean distance of the cylinder $ M $. Also, we denote by $ \Vol $ the volume generated by $ d_M $ on $ M $, and by $ \|\cdot\| $ the Euclidean norm of $ \Rr^2 $. Finally, we denote by $ \pi_s $ and $ \pi_\theta $ the projections defined by $ \pi_s(s,\theta)=s $ and $ \pi_{\theta}(s,\theta)=\theta $ for $ (s,\theta) \in M $. A curve $ \Gamma \subset M $ is called a \emph{horizontal segment} if $ \pi_\theta(\Gamma)=const. $.

Let $ M_P = \bigcup^{n-1}_{i=0} (s_i,s_{i+1}) \times (-\pi/2,\pi/2) \subset M $. We associate to each element $ (s,\theta) \in M_P $ the unit vector $ v $ of $ \Rr^2 $ with base point $ s $ making an angle $ \theta $ with the inner normal to $ \partial P $ at $ s $. Such a normal is not defined at the vertices of $ P $, which is the reason for not including the set $ \bigcup^{n-1}_{i=0} \{s_i\} \times (-\pi/2,\pi/2) $ in $ M_P $. Each pair $ (s,\theta) \in M_P $ specifies the state of the particle immediately after a collision with $ \partial P $: the collision point is given by $ s $ and the velocity is given by the unit vector $ v $. 
 


Given a particle in the state $ (s,\theta) \in M_P $, 
we define $ g_P(s,\theta) $ to be the next point of collision of the particle with $ \partial P $. Let $ Y_P = \{s_{0},\ldots,s_{n-1}\} \times (-\pi/2,\pi/2) $, and let $ N_P = g^{-1}_P(Y_P) $. The \emph{standard billiard map for the polygon $ P $} is the map $ \Phi_P \colon M_P \setminus N_P \to M_P $ whose image $ \Phi_{P}(s,\theta) \in M_P $ corresponds to the state of the particle right after the collision at $ g_P(s,\theta) $ when the reflection law is the specular one. We denote by $ h_P(s,\theta) $  the angle of the particle after the collision. Hence, $ \Phi_P(s,\theta)=\left(g_P(s,\theta),h_P(s,\theta)\right) $. It is not difficult to see that $ \Phi_P $ is piecewise analytic. For a detailed definition of the map $ \Phi_P $, we refer the reader to~\cite{cm06}.


\subsection{Contracting reflection laws}
A reflection law is a function 
\[ 
f \colon (-\pi/2,\pi/2) \to (-\pi/2,\pi/2). 
\] 
For instance, the specular reflection law corresponds to the identity function $ f(\theta)=\theta $. Given a reflection law $ f $, denote by $ R_{f} \colon M_P \to M_P $ the map $ R_{f}(s,\theta)=(s,f(\theta)) $.
The \emph{billiard map for the polygon $ P $ with reflection law $ f $} is the transformation $\Phi_{f,P} \colon M_P \setminus N_P \to M_P $ defined by
\[
\Phi_{f,P} = R_{f} \circ \Phi_{P} = \left(g_P(s,\theta),f \circ h_P(s,\theta)\right). 
\]
Note that $ \Phi_{f,P} $ is injective if and only if $ f $ is, and that $ \Phi_{f,P} $ is a $ C^k $, $ k>0 $ diffeomorphism onto its image of if and only if $ f $ is. 

The differential $ d_{x} \Phi_{f,P} $ is given by~\cite[Section~2.5]{MDDGP13}
\begin{equation}
\label{eq:differential}	
d_{x} \Phi_{f,P} 
= 
- \begin{pmatrix} 
\dfrac{\cos \theta}{\cos(h_{P}(s,\theta))} & \dfrac{t_P(s,\theta)}{\cos(h_{P}(s,\theta))} \\[15pt]
0 & f'(h_{P}(s,\theta))
\end{pmatrix}.
\end{equation}
 	
\begin{defn}
A reflection law $ f $ is called \emph{contracting} if $ f $ is of class $ C^1 $, $ f(0)=0 $ and $ \lambda(f) := \sup\{|f'(\theta)| \colon \theta \in (-\pi/2,\pi/2) \} < 1 $. 
\end{defn}

The simplest example of a contracting reflection law is $ f(\theta) = \sigma \theta $ with $0 < \sigma < 1$ (see Fig.~\ref{fig:reflection law}). This law was considered in several papers~\cite{arroyo09,arroyo12,MDDGP12,markarian10}. 

\begin{figure}[t]
\begin{center}
\includegraphics[scale=0.6]{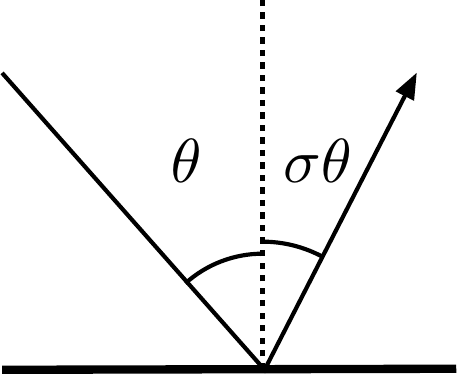}
\caption{Linear contracting reflection law $f(\theta)=\sigma\theta$, with $0<\sigma<1$.}
\label{fig:reflection law}
\end{center}
\end{figure}

We denote by $ \mathcal{R} $ the set of all contracting reflection laws. It is easy to verify that $ \mathcal{R} $ is a Banach space with the norm $ \lambda(f) $. We denote by $ \mathcal{R}^k $, $ k \ge 1 $ the set all contracting reflection laws that are $ C^k $ diffeomorphisms onto their images. 
The reflection law $ f \equiv 0 $ is denoted by $ 0 $. 


In order to apply Pesin's theory to $ \Phi_{f,P} \colon M_P \setminus N_P \to M_P $ and to establish the existence of stable and unstable local manifolds, $ \Phi_{f,P} $ has to be a $ C^2 $ diffeomorphism onto its image $ \Phi_{f,P}(M_P \setminus N_P) $, which is the case if $ f \in \mathcal{R}^2 $.


\subsection{Slap maps}
When $ f = 0 $, the billiard trajectories in $ P $ are all orthogonal to $ \partial P $ after every collision. Thus, the image of the map $ \Phi_{0,P} $ is a subset of the segment $ S^1 \times \{0\} $. If $ f \in \RR $ and $ \lambda(f) $ is sufficiently small, then $ \Phi_{f,P} $ can be considered as a small perturbation of $ \Phi_{0,P} $. Indeed, the two maps have the same domain $ M_P \setminus N_P $, and from the definition of $ \Phi_{f,P} $ and the expression of $ d_x \Phi_{f,P} $, it follows that $ \Phi_{f,P} $ and $ \Phi_{0,P} $ are $ \lambda(f) $-close in the $ C^1 $ topology. 

We now introduce a 1-dimensional map related to $ \Phi_{0,P} $. First, let $ I_P = \bigcup^{n-1}_{i=0} (s_i,s_{i+1}) \subset S^1 $, 
and define $ F_P \colon I_P \to S^1 $ by $ F_P(s) = g_P(s,0) $ for all $ s \in I_P $. 
The map $ F_P $ is related to $ \Phi_{0,P} $, since $ \Phi^n_{0,P}(s,0) = (F^n_P(s),0) $ for all $ (s,0) \in M_P \setminus N_P $ and all $ n \in \Nn $. Moreover, $ F_P $ is affine and strictly decreasing on each connected component of $ I_P $. For this reason, $ F_P $  admits a unique extension to the whole $ S^1 $ that is left\footnote{We might have as well chosen the extension to be right continuous.} continuous at each point $ s_0,s_1,\ldots,s_{n-1} $. We denote such an extension by $ \psi_P \colon S^1 \to S^1 $, and call it the \emph{slap map} of $ P $. See Fig.~\ref{fig:slap}.

\begin{figure}[t]
\begin{center}
\includegraphics[scale=0.15]{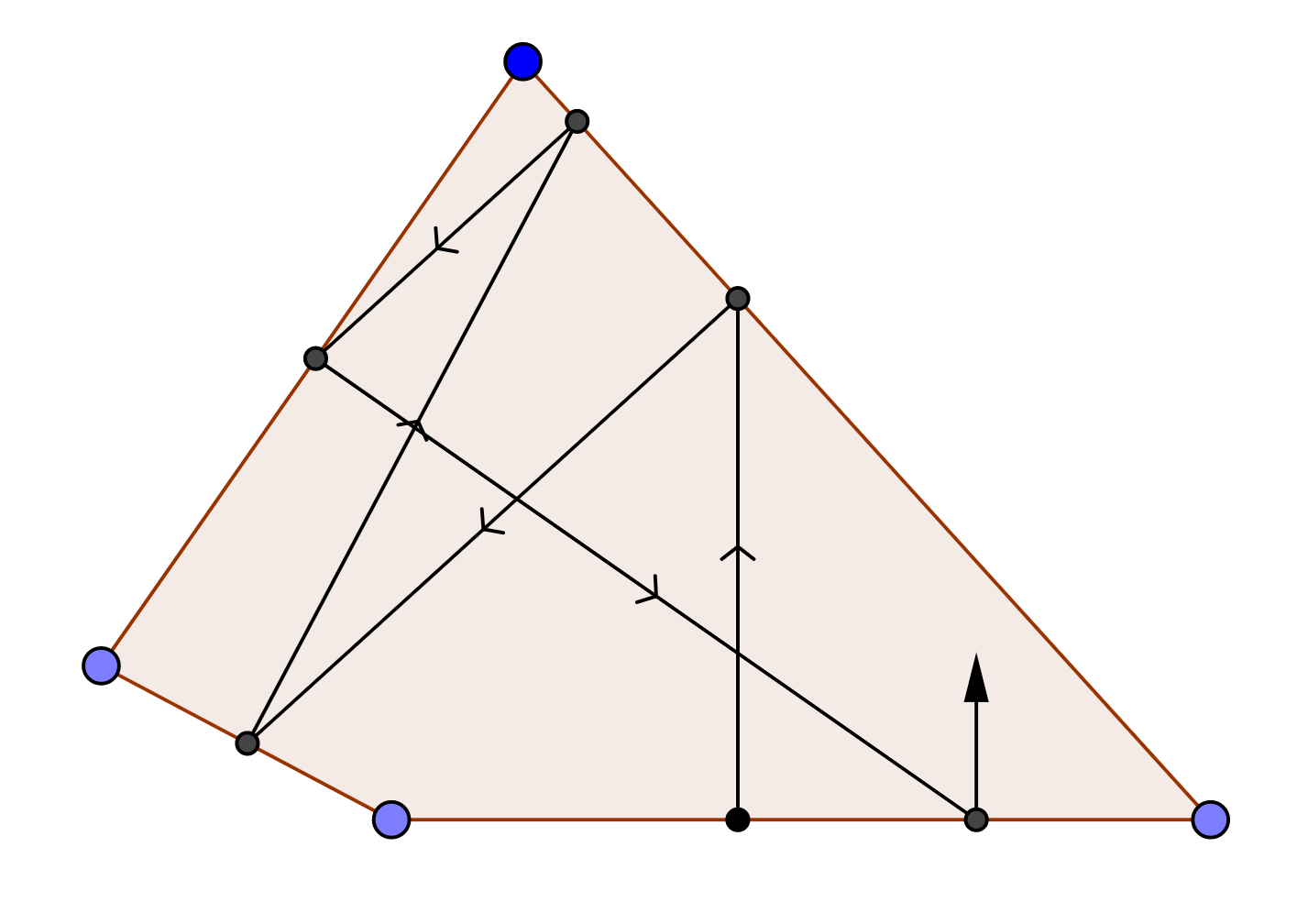}
\caption{Billiard trajectory of a slap map.}
\label{fig:slap}
\end{center}
\end{figure}

The \emph{singular set} of a piecewise expanding map is the set of point where the map does not have continuous second derivatives. 
It is not difficult to see that if $ \psi_P $ is analytic at $ s_i $ if and only if $ s_i $ is a vertex with an acute internal angle. In fact, in that case, $ \psi_P(s_i)=s_i $. Accordingly, the singular set $ S_P $ of $ \psi_P $ is the set of all non-acute vertices of $ P $. 


\section{Hyperbolic polygonal billiards}
\label{se:hyperbolic}


\subsection{Hyperbolic attractors}


Let $ P $ a polygon. For every $ f \in \RR $ with $ f \neq 0 $, define 
\[ 
K_{f,P} = \left\{(s,\theta) \in M_P \colon |\theta| < \lambda(f) \dfrac{\pi}{2} \right\}.
\]
For the special case $ f=0 $, define $ K_{0,P} = \bigcup^{d-1}_{i=0} (s_i,s_{i+1}) \times \{0\} $, and conventionally choose the boundary of $ K_{0,P} $ to be the empty set, that is, $ \partial K_{0,P} = \emptyset $. Since $ f $ is a contraction, the set $ K_{f,P} $ is forward invariant 
\[ 
\Phi_{f,P}(K_{f,P} \setminus N_P) \subset K_{f,P}.
\] 
From now on, we will focus our attention to the restriction of $ \Phi_{f,P} $ to $ K_{f,P} \setminus N_P $,
which by abuse of notation we still denote by $ \Phi_{f,P} $.

Not every every element of $ K_{f,P} \setminus N_P $ can be iterated indefinitely due to the set $ N_P $. The set of all elements of $ K_{f,P} \setminus N_P $ with positive semi-orbit is
\[
K^{+}_{f,P} := \left\{(s,\theta) \in K_{f,P} \colon \Phi_{f,P}(s,\theta) \notin N_P \,\, \forall n \ge 0\right\}.
\] 
Then the maximal forward invariant set of $ \Phi_{f,P} $ is
\[
D_{f,P} :=\bigcap_{n \ge 0} \Phi^{n}_{f,P}(K^{+}_{f,P}).
\]
Note that if $ f \in \RR^1 $, then $ D_{f,P} $ is also the maximal invariant set of $ \Phi_{f,P} $, meaning that $ \Phi^{-1}_{f,P}(D_{f,P}) = D_{f,P} $. Following~\cite{Pesin92}, we call 
\[ 
\Lambda_{f,P} := \overline{D}_{f,P} 
\] 
the \emph{attractor} of $ \Phi_{f,P} $, and 
\[ 
N^+_{f,P}:=(N_P \cap K_{f,P}) \cup \partial K_{f,P} 
\] 
the \emph{singular set} of $ \Phi_{f,P} $. 

In~\cite{MDDGP13,markarian10}, it was proved that for every $ f \in \RR^1 $ and every polygon $ P $, the set $ D_{f,P} $ has a weak form of hyperbolicity called \emph{dominated splitting}. In this paper, 
we are interested in the case when $ D_{f,P} $ is a \emph{hyperbolic set}, that is, when the tangent space of $ K_{f,P} $ at each point $ x \in D_{f,P} $ splits into complementary invariant subspaces $ E^{s}(x) $ and $ E^{u}(x) $ that are uniformly contracted and expanded by the differential of $ \Phi_{f,P} $. 

\begin{defn} The attractor $ \Lambda_{f,P} $ is called \emph{hyperbolic} if $ D_{f,P} $ is a hyperbolic set.
\end{defn}

\begin{defn}
\label{de:facing}
A polygon $ P $ has \emph{parallel sides facing each other}\footnote{This notion is not exactly equal to the one given in~\cite{MDDGP13}. According to this definition arbitrarily small perturbations of a polygon without parallel sides facing each other  may have parallel sides facing each other.} if there exist parallel sides $ L_1 $ and $ L_2 $ of $ P $ and points $ q_1 $ and $ q_2 $ contained in the interior of $ L_1 $ and $ L_2 $, respectively, such that the segment joining $ q_1 $ and $ q_2 $ is contained in $ P $, intersects only the sides $ L_1 $ and $ L_2 $ of $ P $, and is perpendicular to both $ L_1 $ and $ L_2 $. 
\end{defn}  

The following proposition is proved in~\cite[Proposition~3.2 and Corollary~3.4]{MDDGP13}. 

\begin{proposition}
\label{pr:hyperbolicity}
Suppose that $ f \in \RR^1 $. Then $ \Lambda_{f,P} $ is hyperbolic if and only if $P$ does not have parallel sides facing each other. Moreover, if $ \Lambda_{f,P} $ is hyperbolic, then the unstable direction $ E^u $ coincides with the horizontal direction $ \theta = const. $ at every point of $ D_{f,P} $. 
\end{proposition}


\begin{remark}
Note that the horizontal direction is always invariant even if $ D_{f,P} $ is not hyperbolic. This peculiar property is a consequence of the fact that the angle formed by two trajectories bouncing off the same side of the polygon does not change after the collision no matter how the reflection law $ f \in \RR $ is chosen.
\end{remark}


\subsection{Pesin Theory and SRB measures}
\label{su:pesin}
In this subsection, we recall basic results on the existence of local stable and unstable manifolds for the billiard map $ \Phi_{f,P} $ and the definition of SRB measure. We assume $ f \in \RR^2 $. 

Let
\begin{multline*}
N^{-}_{f,P} = \left\{x \in K \colon \exists y \in N^{+}_{f,P} \text{ and } y_n \in K_{f,P} \setminus N^{+}_{f,P} \right. \\ \left. \text{ such that } y_n \to y \text{ and } \Phi_{f,P}(y_n) \to x \right\}.
\end{multline*}
The set $ N^{-}_{f,P} $ can be thought of as `singular set' for the inverse map $ \Phi^{-1}_{f,P} $. Next, for every $ \epsilon>0 $ and every $ l \in \Nn $, define 
\begin{align*}
D^{+}_{f,P,\epsilon,l} & = \left\{x \in \Lambda_{f,P} \cap K^+_{f,P} \colon d_{M}(\Phi^n_{f,P}(x),N^+_{f,P}) \ge \dfrac{e^{-\epsilon n}}{l} \;\;\; \forall n \ge 0 \right\}, \\
D^{-}_{f,P,\epsilon,l} & = \left\{x \in D_{f,P} \colon d_{M}(\Phi^{-n}_{f,P}(x),N^-_{f,P}) \ge \dfrac{e^{-\epsilon n}}{l} \;\;\; \forall n \ge 0 \right\}, \\
D^{0}_{f,P,\epsilon,l} & = D^{-}_{f,P,\epsilon,l} \cap D^{+}_{f,P,\epsilon,l},
\end{align*}
and 
\[
D^{\pm}_{f,P,\epsilon} = \bigcup_{l \ge 1} D^{\pm}_{f,P,\epsilon,l}, \qquad D^{0}_{f,P,\epsilon} = D^{-}_{f,P,\epsilon} \cap D^{+}_{f,P,\epsilon}.
\]
The sets $ D^{0}_{f,P,\epsilon,l} $ play the role of the regular sets 
in the Pesin theory for smooth maps~\cite{Pesin76}. 

\begin{defn}
\label{de:regular}
The attractor $ \Lambda_{f,P} $ is called \emph{regular} if there exists $ \epsilon_0>0 $ such that $ D^0_{f,P,\epsilon} \neq \emptyset $ for every $ 0<\epsilon<\epsilon_0 $. 
\end{defn}


If $ \Lambda_{f,P} $ is hyperbolic and regular, then the Pesin theory for maps with singularities~\cite{KS86} guarantees the existence of an $ \epsilon>0 $ such that a local stable manifold $ W^{s}_{loc}(x) $ exists for all $ x \in D^{+}_{f,P,\epsilon,l} $, and a local unstable manifold $ W^{u}_{loc}(x) $ exists for all $ x \in D^{-}_{f,P,\epsilon,l} $ (see~\cite[Proposition~4]{Pesin92}). The local manifolds $ W^{s}_{loc}(x) $ and $ W^{u}_{loc}(x) $ are $ C^{2} $ ($ C^k $ if $ f \in \mathcal{R}^k $, $ k \ge 2 $) embedded submanifolds whose tangent subspaces at $ x $ are equal to the stable subspace $ E^{s}(x) $ and the unstable subspace $ E^{u}(x) $, respectively. The size of these manifolds depends on the constants $ \epsilon $ and $ l $, and they form two transversal invariant laminations. Finally, we observe that for the billiard map $ \Phi_{f,P} $, the local unstable manifolds are horizontal segments.

\begin{defn}
\label{de:srb}
Suppose that $ \Lambda_{f,P} $ is hyperbolic and regular. Let $ \epsilon>0 $ be such that $ W^{s}_{loc}(x) $ exists for all $ x \in D^{+}_{f,P,\epsilon,l} $, and $ W^{u}_{loc}(x) $ exists for all $ x \in D^{-}_{f,P,\epsilon,l} $. An invariant Borel probability measure $ \mu $ on $ \Lambda_{f,P} $ is called \emph{SRB} if $ \mu(D^{0}_{f,P,\epsilon})=1 $, and the conditional measures of $ \mu $ on the local unstable manifolds of $ \Phi_{f,P} $ are absolutely continuous with respect to the Riemannian volume of the local unstable manifolds. 
\end{defn}

The precise meaning of `the conditional measures of $ \mu $ on the local unstable manifolds of $ \Phi_{f,P} $ are absolutely continuous with respect to the Riemannian volume of the local unstable manifolds' in the previous definition requires a technical explanation that can be found in~\cite{Pesin92}. We remark that what we call here an SRB measure is essentially what is called an \emph{invariant Gibbs $u$-measure} in~\cite{Pesin92}.

In Section~\ref{sec:ergodic stability}, we introduce Condition~(*), and prove that if a polygon $ P $ satisfies it, then the map $ \Phi_{f,Q} $ admits SRB measures for all polygons $ Q $ sufficiently close to $ P $ 
and for all $ f \in \RR^2 $ with $ \lambda(f) $ sufficiently small. We remark that the same result under the additional condition $ f'>0 $ follows from a general result we obtained in~\cite{lxds}. 

\subsection{Expanding slap maps}

 If $ P $ does not have parallel sides facing each other, then $ \psi_P $ is a piecewise expanding map. This means that there exists $ \sigma>1 $ such that $ |\psi'_P|>\sigma $.

By a well-known result of Lasota and Yorke~\cite{LaYo}, piecewise expanding maps have absolutely continuous invariant probability measures (for short acips). The theory of piecewise expanding maps applied to slap maps gives the following.

 
\begin{theorem}
\label{th:slap acip}
If $P$ is polygon without parallel sides facing each other, then there exist subsets $ A_1,\ldots,A_k $ of $ S^1 $ and ergodic acips $ \nu_{1},\ldots,\nu_{k} $ of $ \psi_P $ with bounded variation densities such that  that even though the results listed
\begin{enumerate}
\item $ S^1 = A_1 \cup \cdots \cup A_k $ and $ A_i \cap A_j = \emptyset $ for all $ i \neq j $; 
\item $ \psi^{-1}_P(A_i)=A_i $, $ \nu_i(A_i)=1 $ and $ \psi_P|_{A_i} $ is ergodic with respect to $ \nu_i $ for every $ i = 1,\ldots,k $;
\item for each $ i=1,\ldots,k $, there exist disjoint subsets $ A^1_i,\ldots,A^{n_i}_i $ such that for all $ i,j $,
\begin{enumerate}
\item $ A_i = A^1_i \cup \cdots \cup A^{n_i}_i $;
\item each $ A^j_i $ is $ \psi^{n_i}_P $-invariant;
\item $ \psi^{n_i}_P|_{A^j_i} $ with the normalized restriction of $ \nu_i $ to $ A^j_i $ is exact;
\item $ \supp \nu_{i} $ 
consists of finitely many pairwise disjoint intervals;
\item every open subset of $ \supp \nu_i $ contains two periodic points of $ \psi_P $ whose periods have great common divisor equal to $ n_i $. In particular, the periodic points of $ \psi_{P} $ are dense in $ \supp \nu_i $;
\end{enumerate}
\item the union of the basins of $ \nu_{1},\ldots,\nu_{k} $ has full Lebesgue measure in $ S^1 $.
\end{enumerate}
\end{theorem}

\begin{proof} 
The results cited in the references below are proved for maps of the interval $ [0,1] $. However, they continue to hold for maps of the unit circle. The existence of a finite number of ergodic acips of $ \psi_P $ and Parts~(1), (2) and (3a)-(3d) follow from the general theory of piecewise expanding maps~\cite[Theorem~7.2.3 and Theorem~8.2.2]{BG97}. Part~(3e), which is proved in~\cite[Theorem~3.14 and Proposition~3.15]{lxds2}. For a proof of Part~(4), see \cite[Corollary~3.14]{Viana}.   
\end{proof}

We call the sets $ A_1,\ldots,A_k $ the \emph{ergodic components} of $ \psi_P $, and we call the sets $ A^1_i,\ldots,A^{n_i}_i$ the \emph{exact components} of $ A_i $.


\section{Main results}
\label{sec:ergodic stability}


\subsection{Condition~(*)}
The standing assumption of Theorem~\ref{th:main} is a condition on the orbits of the points in the singular set $ S_P $ of the slap map $ \psi_P $. 

\begin{defn}
A polygon $ P $ satisfies Condition~(*) 
if for every $ s \in S_P $, the forward orbits of 
\[ 
\psi_P(s^+):=\lim_{t \to s^+} \psi_{P}(t) \quad \text{and} \quad\psi_P(s^-):=\lim_{t \to s^-} \psi_{P}(t) 
\] 
do not contain elements of $ S_P $ or periodic points of $ \psi_P $. 
\end{defn}

Since $ S_P $ consists only of non-acute vertices of $ P $, and acute vertices of $ P $ are fixed points of $ \psi_P $, Condition~(*) can be equivalently formulated as follows: 
for 
every non-acute vertex $ s $ of $ P $, the forward orbits of $ \psi_{P}(s^+) $ and $ \psi_{P}(s^-) $ do not visit any vertex of $ P $ and do not contain any periodic point of $ \psi_P $. Examples of orbits that do not satisfy the first part and the second part of Condition~(*) are depicted in Figs~\ref{fig:orthogonal connection} and~\ref{fig:pre periodic}, respectively. 

\begin{remark}
In~\cite[Section~3]{lxds2}, we introduced a condition for general piecewise expanding maps of the interval called Condition~(*) as well. When specialized to slap maps, that condition becomes Condition~(*) above. We also remark that in~\cite{MDDGP13}, we introduced a condition called `no orthogonal vertex connections', which is exactly the first part of Condition~(*). 
\end{remark} 


\begin{figure}[h]
\begin{center}
\includegraphics[scale=0.4]{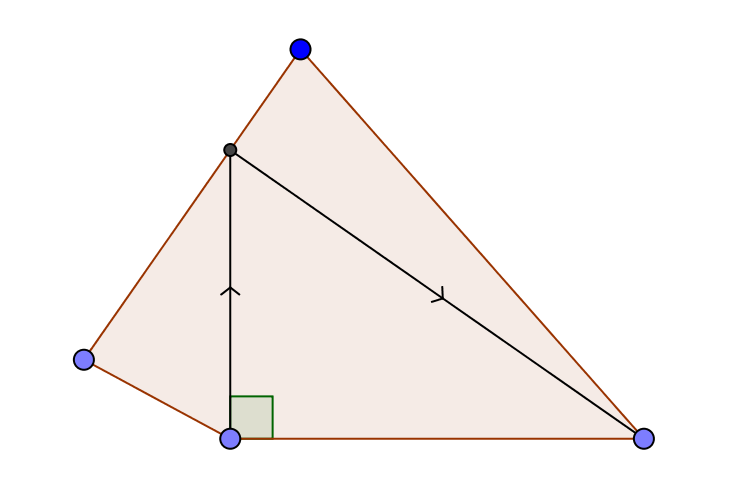}
\caption{Polygon that does not satisfy Condition (*): a non-acute vertex has a forward orbit ending at another vertex.}
\label{fig:orthogonal connection}
\end{center}
\end{figure}

\begin{figure}[h]
\begin{center}
\includegraphics[scale=0.4]{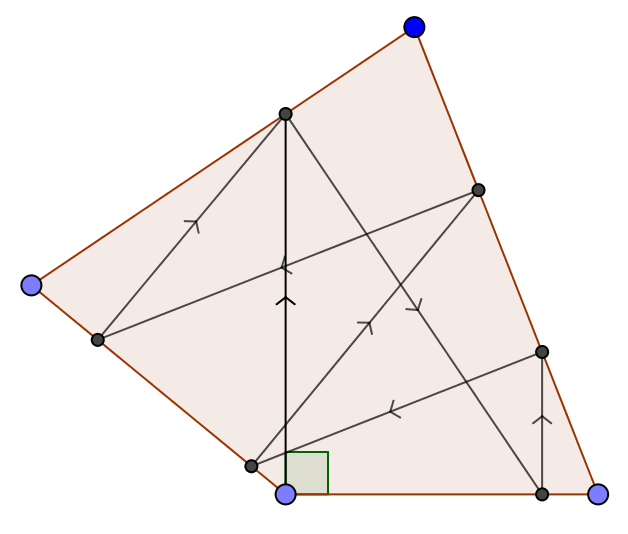}
\caption{Polygon that does not satisfy Condition (*): a non-acute vertex has a forward orbit that is eventually periodic.}
\label{fig:pre periodic}
\end{center}
\end{figure}


Deforming a polygon with $n$ sides by an orientation preserving similar transformation, i.e., dilations, translations and rotations, does not change its billiard dynamics. From this perspective, two polygonal billiards are equivalent if the corresponding polygons are similar to each other. Similarity is an equivalence relation in the set of all polygons with $n$ sides. The quotient of the set of convex polygons with $n$ sides by this equivalence relation is called the \textit{moduli space of convex polygons with $n$ sides}, which we denote by $ \PP_n $. Let $\Pp^1$ and $\Pp^2$ denote the real projective line and real projective plane, respectively. The proof of the following proposition can be found in \cite[Theorem 5.1]{MDDGP13}.

\begin{proposition}
The moduli space $ \PP_n $ is diffeomorphic to an open semialgebraic subset of $\Pp^1\times (\Pp^2)^{n-3}\times \Pp^1$, and it is a manifold of dimension $2n-4$.
\end{proposition}

It is clear that if a convex polygon $P$ satisfies Condition (*), then all convex polygons in the equivalence class $[P]\in\PP_n$ share the same condition. We denote by $ \PP^*_n $ the subset of convex polygons of $ \PP_n $ satisfying Condition (*). Hereafter, in order to simplify the presentation, we also call an element in $ \PP_n $ a convex polygon. As $ \PP_n $ is diffeomorphic to a product of real projective spaces, it carries a topology given by an induced metric $d$ and also an induced measure $m$. An adaptation of the proof of \cite[Proposition 5.3]{MDDGP13} gives the following result.

 \begin{theorem}\label{le:generic}
$\PP^*_n$ is a full measure residual subset of $\PP_n$.
\end{theorem}
\begin{proof}

Given a convex polygon $P$ with $n$ sides, denote by $e_1,\ldots, e_n$ its edges,  by $\ell_i$ the line
supporting $e_i$, and by  $\pi_{j,i}:\ell_i\to\ell_j$
the projection of $\ell_i$ onto $\ell_j$ along the normal direction of $\ell_i$. If $(a_i:b_i:c_i)$ denote the projective coordinates of the line $\ell_i$, then
\begin{equation}\label{orth:proj}
\pi_{j,i}(x,y)=\left(x-  \frac{ a_j x + b_j y + c_j }{a_ia_j+b_ib_j}\, a_i, \, 
y- \frac{ a_j x + b_j y + c_j }{a_ia_j+b_ib_j}\, b_i \,\right) \;.
\end{equation}
Of course, the projection is well defined whenever $\ell_i$ and $\ell_j$ are not perpendicular.
We say that $P$ has a {\em singular connection of order $k\geq 1$} if there is a sequence of edges  
$e_{i_0},e_{i_1},\ldots, e_{i_{k}}$ of $P$ and a sequence of points $p_0,p_1,\ldots, p_k$ of $P$, not all equal, such that:
\begin{enumerate}
\item[(a)] $p_0\in e_{i_0}$ is a vertex of $P$;
\item[(b)] $p_{j}= \pi_{i_{j},i_{j-1}}(p_{j-1})$, for every $j=1,\ldots, k$;
\item[(c)] the polygonal line
$p_0 p_1\ldots  p_{k}$ is contained in $P$;
\item[(d)] either $p_k$ is a vertex of $P$ or $p_k=p_j$ for some $j=1,\ldots,k-1$.
\end{enumerate}
In the previous condition, if $p_k$ is a vertex of $P$, then the singular connection is called \textit{vertex connection}. Otherwise, if $p_k=p_j$ for some $j=1,\ldots,k-1$, then the singular connection is called \textit{periodic connection}.
We denote by $\mathscr{S}_n^k\subseteq \PP_{n}$ the subset of polygons with singular connections of order $k$ and the subset $\mathscr{S}_n^\infty$ to be the union of all $\mathscr{S}_n^k$ over $k\geq1$. Clearly, $\PP_n\setminus \mathscr{S}_{n}^\infty \subseteq \PP_n^*$.

We claim that the set $\mathscr{S}_n^k$ is a finite union of codimension $1$ closed semialgebraic sets. From this property it is immediate that $\PP^*_n$ is a full measure residual subset of $\PP_n$, thus proving the theorem. By \cite[Proposition 5.3]{MDDGP13}, this property is known to hold when we restrict to vertex connections. In the following we show that the set of periodic connections of order $k$ is also a finite union of codimension $1$ closed semialgebraic sets. 

If the edges $e_{i_0},\ldots, e_{i_{k}}$ and the points $p_0,\ldots, p_{k}$ define a periodic connection, then $p_k=p_j$ for some $1\leq j\leq k-1$. Whence,
\begin{equation} \label{ovcp:eq}
\pi_{i_k,i_{k-1}}\circ\cdots\circ \pi_{i_1,i_0}(p_0)=\pi_{i_j,i_{j-1}}\circ\cdots\circ \pi_{i_1,i_0}(p_0),
\end{equation} 
where $p_0$ is a vertex of $P$.

Because \eqref{orth:proj} are rational functions
of the coordinates $(a_i:b_i:c_i)$ of the polygon's edges, equation 
 ~(\ref{ovcp:eq}) is also rational in these coordinates. Re\-du\-cing to a common denominator, and eliminating it, equation~(\ref{ovcp:eq}) becomes polynomial
in the projective coordinates of $\PP_n\subseteq \Pp^1\times(\Pp^{n-3})\times\Pp^1$.
Notice that ~(\ref{ovcp:eq}) is a system of two equations
which reduces to a single equation (non-identically zero) because the projection points  lie on a line.
Given an itinerary $\omega=\{i_0,i_1,\ldots,i_k\}$ with no repeated consecutive symbols, $1
\leq j\leq k-1$ and a vertex $p_0\in\ell_{i_0}$, the set $\Sigma_{\omega,j,p_0}$ of convex polygons defined by the equation~ (\ref{ovcp:eq}) is 
a codimension one closed 
subset of $\PP_n$.
Thus $\mathscr{S}_{n}^k$ is a finite union of such sets, as we wanted to show. 
\end{proof}

%
%

\begin{lemma}
\label{le:parallel}	
If $ P \in \PP^*_n $, then there exists $ \delta>0 $ such that if $ Q \in \PP_n $ and $ d(Q,P)<\delta $, then $ Q $ does not have vertices with internal 
right angle and parallel sides facing each other.
\end{lemma}

\begin{proof}
Condition~(*) implies that no segment contained in $ P $ with endpoints on two sides of $ P $ can be orthogonal to both of them. In particular, no vertex of $ P $ can have right internal angle, and $ P $ does not have parallel sides facing each others. The same is true for every polygon $ Q \in \PP_n $ sufficiently close to $ P $ in the metric $ d $. 
\end{proof}








\subsection{Main results}

Denote by $ \SRB(\psi_{P}) $ the set of all the ergodic acips of $ \psi_P $, and by $ \SRB(\Phi_{f,Q}) $ the set of all the ergodic SRB measures of $ \Phi_{f,Q} $. 
We can now give a precise formulation the main results of the paper.

\begin{theorem}
\label{thm:main2}
Given a polygon $P \in \PP^*_n$, there exists $\delta>0$ such that for every $f\in \RR^2$ with $\lambda(f)<\delta$ and for every $Q\in\PP_n$ with $ d(Q,P)<\delta $, there is a bijection $ \Theta_{f,Q} \colon\SRB(\psi_{P})\to \SRB(\Phi_{f,Q}) $. Moreover, 
\begin{enumerate}
\item the supports of the measures in $ \SRB(\Phi_{f,Q}) $ are pairwise disjoint,
\item the cardinality of $ \SRB(\Phi_{f,Q}) $ is less than or equal to $ n $, 
\item for every $ \nu \in \SRB(\psi_P) $, the number of Bernoulli components of $ \Theta_{f,Q}(\nu) $ equals the number of exact components of $ \nu $,
\item the union of the basins of the measures in $ \SRB(\Phi_{f,Q}) $ is a set of full measure-$ \Vol $ in $ M_Q $.
\end{enumerate}
\end{theorem}


Theorem~\ref{thm:main2} and Theorem~\ref{le:generic} yield immediately the following.

\begin{theorem}
\label{th:main3}
There exists a residual and full measure subset $\XX_n $ of $\PP_n$ with the following property: for every $P \in \XX_n$, there is $ \delta>0 $ such that if $ f \in \RR^2$ with $ \lambda(f)<\delta $ and $ Q \in \PP_n $ with $ d(Q,P)<\delta $, then the conclusion of Theorem~\ref{thm:main2} holds for $ \Phi_{f,Q} $. 
\end{theorem}

%
%

\subsection{Outline of the proof of Theorem~\ref{thm:main2}}
To guide the reader through the proof, we now provide an outline of the proof of Theorem~\ref{thm:main2}. The first part consists of several preliminary results presented in Section~\ref{se:preliminary}. The final part of the proof is given in Section~\ref{se:bijection theta}. Section~\ref{se:preliminary} is divided into five subsections. 


\subsubsection{Trapping regions}
The starting point of the proof of Theorem~\ref{thm:main2} is the construction of a strictly forward-invariant set (trapping region) common to all maps $ \psi_Q $ with $ Q $ sufficiently close to a polygon $ P $ satisfying Condition~(*). Indeed, this condition implies that the slap map $ \psi_P $ is piecewise expanding, and permits to construct for every ergodic acip $ \nu $ of $ \psi_P $, a subset $ U(\nu) $ that contains the support of $ \nu $ and that is a trapping region for every slap map $ \psi_Q $ with $ Q $ sufficiently close to $ P $ (see Proposition~\ref{pr:boundaryacip}).

When the reflection law $ f $ is sufficiently close to the constant function $ 0 $, we can think of $ \Phi_{f,P} $ as a small perturbation of $ \Phi_{0,P} $. Since $ \Phi_{0,P}(s,0) = (\psi_P(s),0) $ for $ s \in S^1 $, by embedding $ S^1 $ in $ S^{1} \times (-\pi/2,\pi/2) $, we can then think of the 2-dimensional map $ \Phi_{f,P} $ as a small perturbation of the 1-dimensional map $ \psi_P $, provided that $ f $ is sufficiently close to the function 0. This allows us to show that the subset $ W(\nu) := U(\nu) \times (-\lambda(f) \pi/2,\lambda(f) \pi/2) $ is a trapping region for all maps $ \Phi_{f,Q} $ `sufficiently close to $ \psi_P $ in the sense specified above, i.e., when $ f $ and $ Q $ are sufficiently close to the function 0 and the polygon $ P $, respectively (see Proposition~\ref{pr:trapping3}). 

\subsubsection{Hyperbolicity} Next, we use again the hypothesis that $ P $ satisfies Condition~(*) to show that whenever $ \Phi_{f,Q} $ is sufficiently close $ \psi_P $, the expansion along the horizontal direction (the unstable direction) generated by $ d\Phi_{f,Q} $ is uniform in $ f$ and $ Q $ (see Lemma~\ref{le:expansion}). This is a rather obvious consequence of the closeness of $ \Phi_{f,Q} $ and $ \psi_P $. A less obvious conclusion is that there exists $ \eta>0 $ such that for every $ \Phi_{f,Q} $ sufficiently close $ \psi_P $ and for every horizontal segment $ \Gamma $, 
there exist a subsegment $ \gamma \subset \Gamma $ and $ k \in \Nn $ with the property that $ \Phi^k_{f,Q}(\gamma) $ is a horizontal segment of length $ \eta $ (see Lemma~\ref{le:growth}). Results of this type are sometimes called \emph{growth lemmas}, and are consequences of the fact that the local expansion of a given hyperbolic map prevails over the local complexity generated by the singularities of the map. Proposition~\ref{pr:analytic} provides the essential information for establishing such a relation between the local expansion and the local complexity for the billiard maps $ \Phi_{f,Q} $. 

\subsubsection{Existence of SRB measures} 
The following step is to establish the existence of SRB measures and their basic ergodic properties for maps $ \Phi_{f,Q} $ sufficiently close to $ \psi_P $. We do that by applying to $ \Phi_{f,Q} $ a theorem of Pesin on the existence of SRB measures for a general class of hyperbolic maps with singularities (see Theorem~\ref{th:SRB}). In this theorem, the SRB measures are obtained as limit points of averages of the push forwards of an absolutely continuous measure supported on a local unstable manifold. We prove that the hypotheses of Pesin's theorem are satisfied using our results on the local expansion and local complexity for the maps $ \Phi_{f,Q} $ (see Proposition~\ref{pr:h}).


\subsubsection{Periodic points} A central role in the proof Theorem~\ref{thm:main2} is played by  sets of periodic orbits of $ \psi_P $ with special recurrent properties (see Lemma~\ref{le:uno}). These orbits can be continued to periodic orbits of maps $ \Phi_{f,Q} $ sufficiently close to $ \psi_P $ (see Theorem~\ref{th:fp}). In the process, the properties of the orbits of $ \psi_P $ are passed to their continuations (see Lemma~\ref{le:unomezzo}), permitting to transfer the ergodic properties of the acips of $ \psi_P $ to the SRB measures of $ \Phi_{f,Q} $.

\subsubsection{A criterion for ergodicity} Another important ingredient of the proof of Theorem~\ref{thm:main2} is a criterion for the ergodicity of an SRB measure using the periodic points of $ \Phi_{f,Q} $ (see Theorem~\ref{th:erg-periodic}). Roughly speaking, the criterion says that the set of all points with a local unstable manifold that has a forward iterate intersecting the stable manifold of a periodic point $ x_0 $ has full measure with respect to some ergodic SRB of $ \Phi_{f,Q} $. We also prove a similar criterion that allows us to estimate the number of Bernoulli components of an ergodic SRB (see Proposition~\ref{pr:period}). Both criteria can be easily extended to general hyperbolic maps with singularities. 

\subsubsection{Final part of the proof} 
In the final part of the proof, we first construct a bijection between the set of the ergodic SRB measures of $ \Phi_{f,Q} $ and the set of the ergodic acips of $ \psi_P $, and then prove that the number of Bernoulli components of an ergodic SRB measure of $ \Phi_{f,Q} $ equals the number of exact components of the corresponding acip of $ \psi_P $. This is achieved using the results previously obtained in Section~\ref{se:preliminary}, and considering sets of of periodic points of $ \psi_P $ and $ \Phi_{f,Q} $ with special recurrent properties. The construction of the bijection consists essentially of two steps. In the first step, we demonstrate that for every ergodic acip $ \nu $ of $ \psi_P $, the corresponding trapping region $ W(\nu) $ contains up to set of zero measure the support of a single ergodic SRB measure of $ \Phi_{f,Q} $. The second step is devoted to the proof that the support of every ergodic SRB measure of $ \Phi_{f,Q} $ is contained up to set of zero measure in a single trapping region $ W(\nu) $ corresponding to some ergodic acip $ \nu $ of $ \psi_P $ (see Lemmas~\ref{le:tre}-\ref{le:onto}). In the very last part of the proof, we use a result obtained in a previous work to prove that the union of the basins of the ergodic SRB measures of $ \Phi_{f,Q} $ is a set of full $ \Vol $-measure.

\section{Preliminary results}
\label{se:preliminary}

This section contains the preliminaries results needed to prove Theorem~\ref{thm:main2}.

\subsection{Trapping regions}

Let $ P \in \PP_n $. By Theorem~\ref{th:slap acip}, the support of an ergodic acip $ \nu $ of $ \psi_P $ consists of finitely many pairwise disjoint closed intervals. In~\cite{lxds2}, we obtained a characterization of the boundary points of $ \supp \nu $. When $ P $ satisfies Condition~(*),  such a characterization can be formulated as follows. 

\begin{proposition}
\label{pr:boundaryacip}
Suppose that $ P \in \PP^*_n $, and let $ \nu $ be an ergodic acip of $ \psi_P $. If $ s \in \partial \supp \nu $, then there exist an orbit segment $ \{s_0,\ldots,s_k\} $, $ k \ge 2 $  of $ \psi_P $ and $ 0 < j < k $ such that
\begin{enumerate}
\item $ s_0 \in S_P \cap \operatorname{int}(\supp \nu) $,
\item either $ s_i = \psi^i(s^+_0) $ for every $ 1 \le i \le k $, or $ s_i = \psi^i(s^-_0) $ for every $ 1 \le i \le k $,
\item $ s_i \in \partial \supp \nu $ for every $ 0 < i < k $, 
\item $ s_k \in \operatorname{int}(\supp \nu) $,
\item $ s = s_j $.
\end{enumerate}
\end{proposition}

We call $ \{s_0,\ldots,s_k\} $ a \emph{boundary segment of $ \supp \nu $}.

\begin{remark}
\label{re:separated}
It is not difficult to see that Proposition~\ref{pr:boundaryacip} and Condition~(*) imply that $ \supp \nu_1 $ and $ \supp \nu_2 $ are disjoint for any pair $ \nu_1,\nu_2 $ of distinct ergodic acips of $ \psi_P $.
\end{remark}

In the next proposition, for a given $ P \in \PP^*_n $ and a given ergodic acip $ \nu $ of $ \psi_P $, we construct a trapping region arbitrarily close to $ \supp \nu $ and common to all slap maps $ \psi_Q $ with $ Q $ sufficiently close to $ P $. A similar conclusion was obtained in~\cite[Lemma~4.3]{lxds2} for more general piecewise expanding maps.

Given $ \zeta>0 $, we denote by $ (\supp \nu)_\zeta $ the $ \zeta $-neighborhood of $ \supp \nu $. 

\begin{proposition}
\label{pr:trapping1}
Suppose that $ P \in \PP^*_n $, and let $ \nu_1,\ldots,\nu_m $ be the ergodic acips of the slap map $ \psi_P $. For every $ \zeta>0 $, there exist $ \delta>0 $, $ \tau>0 $ and pairwise disjoint closed sets $ U_1,\ldots,U_m $ of $ S^1 $ such that if $ d(Q,P)<\delta $, then for every $ 1 \le i \le m $, 
\begin{enumerate}
\item $ \supp \nu_i \subset \operatorname{int}(U_i) \subset (\supp \nu_i)_\zeta $, 
\item $ U_i $ is a union of finitely many pairwise disjoint closed intervals whose endpoints are not vertices of $ P $,
\item $ \psi_Q(U_i) \subset \operatorname{int}(U_i) $ and 
$ d_{S^1}(\psi_Q(U_i),\partial U_i) > \tau $.
\end{enumerate}
\end{proposition}

\begin{proof}
We construct the sets $ U_1,\ldots,U_m $ inductively. 

We start with the set $ U_1 $. Its construction is also inductive, and requires $ l $ steps, i.e., as many steps as the number of distinct boundary segments $ \gamma_1,\ldots,\gamma_l $ of $ \supp{\nu_1} $. At the $ k $th step, we enlarge $ \supp{\nu}_1 $ by enlarging the intervals forming $ \supp{\nu}_1 $ whose endpoints lie on the $ k $th boundary segment.

Set $ A_0 = \supp \nu_1 $. Suppose that $ A_{k-1} $ with $ 1 \le k \le l $ is given, and let $ \gamma_k = \{s_0,\ldots,s_{N_k}\} $ be the $ k $th boundary segment of $ \nu_1 $. Once again, we construct $ A_k $ inductively. Set $ B_0=A_{k-1} $. Suppose that $ B_{i-1} $ with $ 1 \le i < N_{k} $ is given. If $ s_{N_k - i} \in \operatorname{int}(B_{i-1}) $ -- which happens when there exists $ k'<k $ such that $ \gamma_{k'} $ and $ \gamma_k $ shares the same point $ s_{N_k - i + 1} $ -- then we set $ B_{i} = B_{i-1} $ . Otherwise, if $ s_{N_k - i} $ is the right endpoint of an interval of $ B_{i-1} $, then choose $ t_{N_k - i} \in (s_{N_k - i},s_{N_k - i}+\zeta) $ satisfying
\begin{enumerate}[(i)]
\item $ [s_{N_k - i},t_{N_k -i}] \cap S_P = \emptyset $,
\item $ [s_{N_k - i},t_{N_k -i}] \cap B_{i-1} = \emptyset $,
\item $ [s_{N_k - i},t_{N_k -i}] \cap (\supp \nu_2 \cup \cdots \cup \supp \nu_m) = \emptyset $,
\item $ \psi_P([s_{N_k - i},t_{N_k -i}]) \subset \operatorname{int}(B_{i-1}) $.
\end{enumerate}
Instead, if $ s_{N_k - i} $ is the left endpoint of an interval of $ B_{i-1} $, then choose $ t_{N_k - i} \in (s_{N_k - i}-\zeta,s_{N_k - i}) $ satisfying conditions analogous to (i)-(iv). Such a $ t_{N_k -i} $ exists, because $ s_{N_k - i} \notin S_P $, $ \psi_P(s_{N_k - i}) \in \operatorname{int}(B_{i-1}) $, and the supports of the acips of $ \psi_P $ are pairwise disjoint. Next, define $ B_i = B_{i-1} \cup [s_{N_k - i},t_{N_k -i}] $. The previous procedure gives the sets $ B_0,\ldots,B_{N_k-1} $. finally, define $ A_k = B_{N_k - 1} $. Once all the sets $ A_0,\ldots,A_l $ have been computed, define $ U_1 = A_l $. 

An almost identical construction produces the sets $ U_2,\ldots,U_m $. 
Assume that $ U_{j-1} $ with $ 2 \le j \le m $ is given, and  construct $ U_j $ by following the same procedure used for $ U_1 $ with the obvious modifications and with Condition~(iii) above replaced by 
\[
[s_i,t_i] \cap \left(U_1 \cup \cdots \cup U_{j-1} \cup \supp \nu_{j+1} \cup \cdots \cup \supp \nu_{m} \right) = \emptyset.
\]
It is easy to see that the sets $ U_1,\ldots,U_m $ obtained this way have the wanted properties for the map $ \psi_P $. In fact, besides $ \psi_P(U_i) \subset \operatorname{int}(U_i) $ for every $ 1 \le i \le m $, a bit more can be derived from the construction  above. Namely, we obtain that there is $ \tau>0 $ such that 
\begin{equation}
\label{eq:contained}	
d_{S^1}(\psi_P(U_i), \partial U_i) > 2\tau \qquad \text{for every } 1 \le i \le m.
\end{equation}

Next, we want to extend the previous conclusion to every map $ \psi_Q $ with $ d(Q,P) $ sufficiently small. To this end, note that if $ d(Q,P) $ is sufficiently small, then there is a natural bijective correspondence between the vertices of $ Q $ and $ P $. So for $ d(Q,P) $ sufficiently small case, denote by $ s_Q $ the vertex of $ Q $ corresponding to the vertex $ s $ of $ P $. Then, for every vertex $ s $ of $ P $,
\[
s_Q \to s \quad \text{as } d(Q,P) \to 0. 
\]
Moreover by Condition~(*) and its consequence that no vertex of $ P $ can have an internal angle equal to $ \pi/2 $ (see Lemma~\ref{le:parallel}), it follows that for every vertex $ s $ of $ P $,
\begin{equation}
\label{eq:convergence}	
\psi_{Q}(s^{\pm}_Q) \longrightarrow \psi_{P}(s^{\pm}) \quad \text{as } d(Q,P) \to 0. 
\end{equation} 

Finally, by construction of $ U_i $ and properties~\eqref{eq:contained} and~\eqref{eq:convergence}, it is not difficult to see that $ \psi_Q(U_i) \subset \operatorname{int}(U_i) $ and $ d_{S^1}(\psi_Q(U_i), \partial U_i) > \tau $ for every $ 1 \le i \le m $ provided that $ d(Q,P) $ is sufficiently small. 
\end{proof}


Recall that $ \Phi_{f,Q} \colon K_{f,Q} \setminus N_Q \to K_{f,Q} $. Next, we show that if a polygon $ Q $ is sufficiently close to $ P \in \PP^*_n $, and $ f $ is a reflection law sufficiently close to 0, then the map $ \Phi_{f,Q} $ has a trapping region close to $ \bigcup^m_{i=1} U_{i} \times \{0\} $ with $ U_1,\ldots,U_m $ being as in Proposition\ref{pr:trapping1}.

\begin{proposition}
\label{pr:trapping3}	
Suppose that $ P \in \PP^*_n $. Given $ \zeta>0 $, let $ \delta>0 $, $ \tau>0 $ and the sets $ U_1,\ldots,U_m $ be as in Proposition~\ref{pr:trapping1}. There exist $ 0 < \delta' < \delta $ and  pairwise disjoint sets $ W_1,\ldots,W_m $ of $ M $ defined by
\[
W_i = U_i \times \left(-\frac{\pi}{2} \lambda(f),\frac{\pi}{2} \lambda(f)\right), \qquad i=1,\ldots,m 
\]
such that if $ \lambda(f)<\delta' $ and $ d(Q,P)<\delta' $, then $ \Phi_{f,Q}(W_i \setminus N_Q) \subset \operatorname{int}(W_i) $ for every $ 1 \le i \le m $.
\end{proposition}

\begin{proof}
Given $ \zeta>0 $, let $ \delta>0 $, $ \tau>0 $ and the sets $ U_1,\ldots,U_m $ be as in Proposition~\ref{pr:trapping1}. Define $ W_i $ as in the statement of the proposition. 
Clearly, the sets $ W_1,\ldots,W_m $ are disjoint because so are $ U_1,\ldots,U_m $, and satisfy Condition~(1).  

Next, recall that $ \phi_{f,Q}=(g_Q,f\circ h_Q) $.
Since $ |f \circ h_Q| < \pi \lambda(f)/2 $, to prove Condition~(2), all we need to show is that there exists $ 0<\delta'<\delta $ such that if $ \lambda(f)<\delta' $ and $ d(Q,P)<\delta' $, then $ g_Q(W_i \setminus N_Q) \subset \operatorname{int}(U_i) $.   
By Part~(3) of Proposition~\ref{pr:trapping1}, that is an immediate consequence of
\[ 
|g_Q(s,\theta)-\psi_Q(s)| \le \tau, \qquad (s,\theta) \in W_i \setminus N_Q.
\]

Let $ (s,\theta) \in W_i \setminus N_Q $. Since $ \psi_Q(s) = g_Q(s,0) $, using the Mean Value Theorem, we obtain
\begin{align*}
|g_Q(s,\theta)-\psi_Q(s)| & = |g_Q(s,\theta)-g_Q(s,\theta)| \\
& \le \sup_{-\pi \lambda(f)/2 < \theta < \pi \lambda(f)/2} \left|\partial_{\theta} g_Q(s,\theta)\right| |\theta|.
\end{align*}
By~\eqref{eq:differential}, $ \partial_{\theta} g_Q(s,\theta) = -t_Q(s,\theta)/\cos (h_Q(s,\theta)) $. The function $ t_{Q} $ is bounded by $ 1 $, because the perimeter of $ Q $ is equal to 1. 
By Lemma~\ref{le:parallel}, if $ \lambda(f) $ and $ d(Q,P) $ are sufficiently small, then $ |h_Q(\cdot,0)| $ is uniformly bounded away from $ \pi/2 $ on $ K_{f,Q} \setminus N_Q $, i.e., there is a constant $ A>0 $ depending only on $ P $ such that $ \cos \circ h_Q > A $ on $ K_{f,Q} \setminus N_Q $. Therefore,
\[
\left|g_Q(s,\theta)-\psi_Q(s)\right| < \frac{|\theta|}{A}<\frac{\pi}{2A} \lambda(f), \qquad (s,\theta) \in W_i \setminus N_Q,
\]
provided that $ \lambda(f) $ and $ d(Q,P) $ are sufficiently small. By further taking a smaller $ \lambda(f) $, we obtain $ \pi/(2A) \lambda(f) \le \tau $. 

By the previous conclusion, we can find $ 0 < \delta' < \delta $ such that if $ \lambda(f) < \delta' $ and $ d(Q,P)<\delta' $, then \[ 
\pi_s \circ \Phi_{f,Q}(W_i \setminus N_Q) = g_Q(W_i \setminus N_Q) \subset \operatorname{int}(U_i).
\] 
\end{proof}


\subsection{Hyperbolicity}
\label{su:hyp}

\begin{defn}
Given $ Q \in \PP_n $, $ f \in \mathcal{R} $ and $ m \in \Nn $, denote by $ \alpha(\Phi^m_{f,Q}) $ and $ \beta(\Phi^m_{f,Q})$ the infimum and the supremum of $ \|d_x \Phi^m_{f,Q} (1,0)^T \| $, respectively,  
over the subset of $ K_{f,Q} $ where $ \Phi^m_{f,Q} $ is differentiable.
\end{defn}

The next lemma says that the horizontal direction is uniformly expanding for $ \Phi_{f,Q} $ whenever $ P \in \PP^*_n $, $ f \in \mathcal{R} $ and $ Q \in \PP_n $ with $ \lambda(f) $ and $ d(Q,P) $ sufficiently small. We emphasize that $ f $ and $ \Phi_{f,P} $ may not be invertible for $ f \in \mathcal{R} $.

\begin{lemma}
\label{le:expansion}
Let $ P \in \PP^*_n $. Then there exist $ \delta>0 $ and $ 1<\alpha_0<\beta_0 $ such that 
\begin{enumerate}
\item if $ f \in \mathcal{R}^1 $ and $ d(Q,P)<\delta $, then $ \Lambda_{f,Q} $ is hyperbolic;
\item if $ f \in \mathcal{R} $ with $ \lambda(f)<\delta $ and $ d(Q,P)<\delta $, then 
\[
 \alpha_0 \le \alpha(\Phi_{f,Q}) \le \beta(\Phi_{f,Q}) \le \beta_0. 
 \]	
\end{enumerate}
\end{lemma}

\begin{proof}
Part~(1). By Lemma~\ref{le:parallel}, every $ Q $ sufficiently close to $ P $ has no parallel sides facing each other. For such a $ Q $, Proposition~\ref{pr:hyperbolicity} guarantees that the attractor $ \Lambda_{f,Q} $ is hyperbolic for every $ f \in \mathcal{R}^1 $.

Part~(2). Let $ f \in \mathcal{R} $, and let $ Q \in \PP_n $. By~\eqref{eq:differential}, we have 
\[
\alpha_{f,Q}(x) := \left\|d_x \Phi_{f,Q} 
\begin{pmatrix} 0 \\ 1 \end{pmatrix}\right\| = \frac{\cos \theta}{\cos(h_Q(s,\theta))}
\]
for every $ x = (s,\theta) \in K_{f,Q} \setminus N^+_{f,Q} $. Denote by $ L(x) $ the line containing the side of $ Q $ where $ s $ lies. Also, denote by $ 0<\omega_{Q}(x)<2\pi $ be the external angle between the lines $ L_Q(x) $ and $ L_Q(\Phi_{f,Q}(x)) $, i.e., the angle of a counterclockwise rotation centered at the point of intersection $ L_Q(x) \cap L_Q(\Phi_{f,Q}(x)) $ mapping $ L_Q(x) $ to $ L_Q(\Phi_{f,Q}(x)) $. A simple computation shows that $ h_Q(s,\theta)=\pi-\omega_{q}(x)-\theta $, 
and so
\[
\alpha_{f,Q}(x) = \frac{\cos \theta}{\cos(\pi-\omega_{Q}(x)-\theta)} = \frac{1}{-\cos \omega_{Q}(x) + \tan \theta \sin \omega_{Q}(x)}.
\]

Now, let $ Q = P \in \PP^*_n $. Since $ P $ satisfies Condition~(*), $ P $ does not have sides parallel facing each other and adjacent sides that are perpendicular (see Lemma~\ref{le:parallel}). It is still possible for $ P $ to have parallel sides, but two consecutive collisions $ x $ and $ \Phi_{f,P}(x) $ at parallel sides never occur. It follows that if $ \lambda(f) $ is sufficiently small, then the angle $ \omega_P(x) $ must satisfy the following property: there exist $ \tau>0 $ depending only on $ P $ and $ \delta_0>0 $ such that if $ \lambda(f)<\delta_0 $, then 
\begin{equation}
\label{eq:omega}	
\omega_{P}(x) \in \left(\frac{\pi}{2}+\tau,\pi-\tau\right) \cup \left(\pi+\tau,\frac{3}{2}\pi - \tau\right) \quad \text{for all } x \in K_{f,P} \setminus N^+_{f,P}.
\end{equation} 
Recall that $ x \in K_{f,P} $ means that $ |\theta|<\pi \lambda(f)/2 $ for $ f \neq 0 $, and $ \theta=0 $ for $ f=0 $. From this, propriety~\eqref{eq:omega} and the expression of $ \alpha_{f,P}(x) $, it follows that there exists $ 0<\delta_1 \le \delta_0 $ and $ 1<\alpha_0<\beta_0 $ such that if $ \lambda(f) < \delta_1 $, then 
\begin{equation}
\label{eq:alpha}	
\alpha_0< \alpha_{f,P}(x) < \beta_0 \qquad \text{for all } x \in K_{f,P} \setminus N^+_{f,P}.
\end{equation}

By~Lemma~\ref{le:parallel}, every $ Q $ sufficiently close to $ P $ does not have parallel sides facing each other and adjacent sides that are perpendicular. From this, it is not difficult to see that there must exist $ 0<\delta \le \delta_1 $ such that~\eqref{eq:omega} with the same $ \tau $, and therefore~\eqref{eq:alpha} with the same $ \alpha_0 $ and $ \beta_0 $ continue to hold for every $ Q \in \PP_n $ with $ d(Q,P)<\delta $ and every $ f \in \mathcal{R} $ with $ \lambda(f)<\delta $. This implies the wanted conclusion.
\end{proof}

%

Recall that $ V_Q = \{s_0,\ldots,s_{n-1}\} $ is the set of the vertices of $ Q \in \PP_n $. For every $ s \in V_Q $ and every $ r>0 $, define
\[
I(s,r) = (s-r,s) \cup (s,s+r) \subset S^1,
\]
and
\[
H_Q(r) = \bigcup_{s \in V_Q} I(s,r) \times (-r,r) \subset M.
\] 

The first conclusion of the next proposition is an obvious consequence of Lemma~\ref{le:expansion}. The second conclusion says,  roughly speaking, that for every $ m \in \Nn $, the map $ \Phi^m_{f,Q} $ is differentiable on a sufficiently small neighborhood of the `vertices' of the polygon provided that $ P $ satisfies~(*) and $ \Phi_{f,Q} $ is sufficiently close to $ \Phi_{0,P} $. In turn, this fact implies that when a sufficiently short curve $ \gamma $ in $ K_{f,Q} $ is iterated forward $ m $ times, it cannot be cut more than once by the singular set of $ \Phi_{f,Q} $.

\begin{proposition}
\label{pr:analytic}
Let $ P \in \PP^*_n $. For every $ \bar{\alpha} > 1 $, there exist $ \delta>0 $, $ m \in \Nn $ and $ r>0 $ such that if $ f \in \mathcal{R} $ with $ \lambda(f)<\delta $ and $ d(Q,P)<\delta $, then 
\begin{enumerate}
\item $ \bar{\alpha} \le \alpha(\Phi^m_{f,Q}) \le \beta(\Phi^m_{f,Q}) \le \beta^m_0 $ with $ \beta_0 $ be as in Lemma~\ref{le:expansion},
\item $ \Phi^m_{f,Q}|_{H_Q(r)} $ is differentiable.
\end{enumerate}
\end{proposition}

\begin{proof}
Let $ P \in \PP^*_n $, and let $ \bar{\alpha}>1 $. By Lemma~\ref{le:expansion}, there exist $ m \in \Nn $ and $ \delta_0>0 $ such that 
\[ 
\bar{\alpha} \le \alpha^m_0 \le \alpha(\Phi^m_{f,Q}) \le \beta(\Phi^m_{f,Q}) \le \beta^m_0
\] 
for all $ f \in \RR $ with $ \lambda(f)<\delta_0 $ and all $ Q \in \PP_n $ with $ d(Q,P)<\delta_0 $. The value of $ m $ will be kept fixed throughout the rest of the proof.

Recall that $ Y_P = V_P \times (-\pi/2,\pi/2) $. Since $ P $ satisfies Condition~(*), it follows that if $ s \in S_P $ (i.e., $ s $ is a singular point of $ \psi_P $), then the forward orbits of $ \psi_{P}(s^+) $ and $ \psi_{P}(s^-) $ do not visit any vertex of $ P $. Also, recall that each vertex of $ P $ in $ V_P \setminus S_P $ is a fixed point of $ \psi_P $. 
Hence, there exists $ r_0>0 $ such that 
$d_{S^1}(V_P,\psi^i_P(I(s,r_0)))>0 $ for every $ 0 \le i \le m $ and every $ s \in V_P $. Equivalently, in terms of the map $ \Phi_{0,P} $, 
\[ 
d_{M}(Y_P,\Phi^i_{0,P}(I(s,r_0) \times \{0\}))>0 
\]
for all $ 0 \le i \le m $ and all $ s \in V_P $.
It is not difficult to see that the conclusion remains valid for every $ Q $ sufficiently close to $ P $. More precisely, there is $ 0 < \delta_1 \le \delta_0 $ such that if $ d(Q,P)<\delta_1 $, then 
\begin{equation}
\label{eq:d}	
d_{M}(Y_Q,\Phi^i_{0,Q}(I(s,r_0) \times \{0\}))>0 \qquad \forall 0 \le i \le m \;\; \forall s \in V_Q.
\end{equation}

Now, arguing as in the proof of Proposition~\ref{pr:trapping3},
one can show that there is $ 0 < r \le r_0 $ such that~\eqref{eq:d} holds even when $ I(s,r_0) \times \{0\} $ is replaced by $ I(s,r) \times \{-r,r\} $, and  
that there is $ 0 < \delta \le \delta_1 $ such that if $ f \in \RR $ with $ \lambda(f)<\delta $ and $ d(Q,P)<\delta $, then 
\[ 
d_{M}(Y_Q,\Phi^i_{f,Q}(I(s,r) \times (-r,r)))>0 
\] 
for all $ 0 \le i \le m $ and all $ s \in V_Q $. By taking $ \delta $ sufficiently small, one can even guarantee that 
\[ 
\Phi^{i}_{f,Q}(I(s,r) \times (-r,r)) \subset S^1 \times (-r,r) 
\] 
for all $ 0 \le i \le m $ and all $ s \in V_Q $. What we have just proved can be reformulated in terms of the set $ H_Q(r) $ as follows. There exist $ r>0 $ and $ \delta>0 $ such that if $ f \in \RR $ with $ \lambda(f)<\delta $ and $ d(Q,P)<\delta $, then 
\[	
d_{M}(Y_Q,\Phi^i_{f,Q}(H_Q(r)))>0 \quad \text{and} \quad \Phi^{i}_{f,Q}(H_Q(r)) \subset S^1 \times (-r,r) 
\]
for every $ 0 \le i \le m $. The first inequality implies that $ \Phi^{m}_{f,Q} $ is differentiable on $ H_{Q}(r) $. Indeed, suppose that the claim was not true. Then there would exist $ 0 \le i < m $ such that $ \Phi^i_{f,Q}(H_{Q}(r)) \cap N^+_{f,Q} \neq \emptyset $ implying $ d_{M}(Y_Q,\Phi^i_{f,Q}(H_{Q}(r)))=0 $, which is impossible.
\end{proof}


The last lemma of this section is a corollary of Proposition~\ref{pr:analytic}. It illustrates yet another consequence of Condition~(*): if $ P \in \PP^*_n $, then every horizontal segment contains a subsegment that reaches a uniform length after $ k $ iterations of every $ \Phi_{f,Q} $ sufficiently close to $ \Phi_{0,P} $ with $ k $ being independent of $ f $ and $ Q $.

\begin{lemma}
\label{le:growth}
Let $P \in \PP^*_n$. Then there exist $ \delta>0 $ and $ \eta>0 $ such that if $f\in\mathcal{R}$, $ Q \in \PP_n $ with $\lambda(f)<\delta$ and $ d(Q,P)<\delta $, and $\Gamma$ is a horizontal segment, then there exist a segment $\gamma\subset\Gamma$ and $k \in \Nn $ with the property that $ \Phi_{f,Q}^k(\gamma) $ is a horizontal segment with $ |\Phi_{f,Q}^k(\gamma)| > \eta $.
\end{lemma}

\begin{proof}

Choose\footnote{This choice is arbitrary, every $ \bar{\alpha}>2 $ will do.} $ \bar{\alpha} = 3 $, and let $ m>0 $, $ r>0 $ and $ \delta>0 $ be as in Proposition~\ref{pr:analytic}. Let $ f \in \RR $ with $ \lambda(f)<\delta $, and let $ Q \in \PP_n $ with $ d(Q,P)<\delta $.

Let $ \Gamma $ be a horizontal segment in $ K_{f,Q} $. Define recursively a sequence of horizontal segments $ \{\Gamma_j\} $ as follows: let $ \Gamma_0 = \Gamma $, and   
let $ \Gamma_{j+1} $ be any interval of maximal length of $ \Phi^m_{f,Q}(\Gamma_j) $ for every $ j \ge 0 $. We claim that $ |\Gamma_j| > r $ for some $ j \ge 0 $. It is easily seen that the conclusion of the lemma with $ \eta=r $ and $ k=mj $ is a direct consequence of our claim. 

To prove the claim, we study separately the two alternatives: 1) $ \Phi^m_{f,Q}(\Gamma_j) $ consists of more than two horizontal segments for some $ j \ge 0 $, and 2) $ \Phi^m_{f,Q}(\Gamma_j) $ consists of one or two horizontal segments for every $ j $. If alternative 1 occurs, 
then there exist $ 1 \le i \le m $ and a segment $ \Gamma' \subset \Phi^{i}_{f,Q}(\Gamma_j) $ with both endpoints on $ Y_Q $. By Proposition~\ref{pr:analytic}, 
\[
\Gamma' \subset \Phi^{i}_{f,Q}(H_Q(r)) \subset S^1 \times (-r,r),
\] 
and so the intersection $ \Gamma' \cap H_{Q}(r) $ contains a segment $ \Gamma'' $ of length $ r $. Since $ \Phi^{m-i}_{f,Q} $ is continuos on $ H_{Q}(r) $ and $ \alpha(\Phi_{f,Q})>1 $,  
\[ 
|\Gamma_{j+1}| \ge |\Phi^{m-i}_{f,Q}(\Gamma'')| > r.
\] 
If alternative 2 occurs, then we clearly have
$ |\Gamma_{j+1}| \ge  |\Gamma_{j}| \bar{\alpha}/2 $, and so 
\[ 
|\Gamma_j| \ge |\Gamma_0| \left(\frac{\bar{\alpha}}{2}\right)^j > r 
\]
for some $ j \ge 0 $. Hence, in both cases, there exists $ j \ge 0 $ such that $ |\Gamma_j| > r $. This completes the proof.
\end{proof}

\subsection{Existence of SRB measures}
\label{su:srb}


Recall that $ \Vol $ denotes the volume of generated by the Riemannian metric $ d_M $ on $ M $. Given a $ C^1 $-curve $ \Gamma \subset K_{f,Q} $, denote by $ |\Gamma| $ and $ \Vol_{\Gamma} $ the length of $ \Gamma $ and the normalized volume of $ \Gamma $ induced by the metric $ d_{M} $, respectively. 
Finally, denote by $ N^+_{f,Q}(r) $ the neighborhood of $ N^+_{f,Q} $ in $ K_{f,Q} $ of radius $ r>0 $. 

\begin{proposition}
\label{pr:h}
If $ P \in \PP^*_n $, then there exists $ \delta>0 $ such that if $ f \in \RR^2 $ with $ \lambda(f) < \delta $ and $ d(Q,P)<\delta $, then the following properties hold:
\begin{enumerate}[(A)]
\item there are positive constants $ C=C(f,Q) $ and $ r_0=r_0(f,Q) $ such that
\[ 
\Vol(\Phi^{-n}_{f,Q}(N^+_{f,Q}(r))) < Cr \qquad \forall n \ge 1 \quad \forall 0<r<r_0,
\] 
\item there is $ r_1=r_1(f,Q)>0 $ such that for every horizontal segment $ \Gamma $, there exists 
$ B=B(f,Q,\Gamma)>0 $ for which 
\[ 
\Vol_{\Gamma}(\Gamma \cap \Phi^{-n}_{f,Q}(N^+_{f,Q}(r))) < B r \qquad \forall n \ge 1 \quad \forall 0<r<r_1.
\] 
\end{enumerate}
\end{proposition}

\begin{proof}
In~\cite[Lemma~4.9 and Theorem~4.15]{lxds2}, we demonstrated that properties~(A) and~(B) are consequences of the \emph{$ m $-step expansion condition} (c.f.~\cite[Inequality~(5.38)]{cm06}): there exists $ m \in \Nn $ such that 
\begin{equation}
\label{eq:m}
\liminf_{\tau \to 0^+} \sup_{\Gamma \in \mathcal{H}(\tau)} \sum_{\gamma \in S_m(\Gamma)} \dfrac{1}{a_m(\gamma)} < 1, 
\end{equation}
where $ \mathcal{H}(\tau) $ is the set of horizontal segments $ \Gamma \subset K_{f,Q} $ with $ |\Gamma|<\tau $, $ S_m(\Gamma) $ is the set of maximal subsegments $ \gamma $ of $ \Gamma $ such that $ \Phi^m_{f,Q}|_{\gamma} $ is differentiable, and $ a_m(\gamma) = \inf_{x \in \gamma} \|d_x\Phi^m_{f,Q}(1,0)^T\| $. 
Accordingly, to prove the proposition, it suffices to show that the $ m $-step expansion condition holds for a proper $ m \in \Nn $. 

Let $ P \in \PP^*_n $, and denote by $ L $ the length of the shortest side (or sides) of $ P $. Choose $ \bar{\alpha}>2 $, and let $ \delta $, $ m $, $ r $ and $ \beta_0 $ be the positive constants as in Proposition~\ref{pr:analytic}. Consider  $ f \in \mathcal{R}^2 $ with $ \lambda(f)<\delta $ and $ Q \in \PP_n $ with $ d(Q,P)<\delta $. If necessary, take a smaller $ \delta $ so that the length of the shortest side (or sides) of $ Q $ is less than $ 3L/2 $. Choose $ \delta =  \beta^{-m}_0 \cdot \min\{r,l\} $, and let $ \Gamma \in \mathcal{H}(\tau) $. 

If $ N^+_{f,Q} \cap \Phi^i_{f,Q}(\Gamma) = \emptyset $ for every $ 0 \le i < m $, then $ \Phi^i_{f,Q}|_{\gamma} $ is differentiable, and $ \Phi^i_{f,Q}(\Gamma) $ consists of a single segment. Thus $ S_m(\Gamma) = \{\Gamma\} $, $ a_{m}(\Gamma) \ge \bar{\alpha} $ and 
\[
\sup_{\Gamma \in \mathcal{H}(\tau)} \sum_{\gamma \in S_m(\Gamma)} \dfrac{1}{a_m(\gamma)} = \frac{1}{\bar{\alpha}} < \frac{1}{2}. 
\]

Now, suppose that $ N^+_{f,Q} \cap \Phi^i_{f,Q}(\Gamma) \neq \emptyset $ for some $ 0 \le i < m $, and let $ 0 \le j < m $ be the smallest $ i $ with such a property. It follows that $ \Phi^j_{f,Q}(\Gamma) $ consists of several disjoint horizontals segments whose total length is less than $ L \beta^{-m+j}_0 < L $, because $ \beta^m_0 $ is the supremum of the expansion along the horizontal direction. Since $ Q $ is convex, and the length of its shortest side is less than $ 3L/2 $, $ \Phi^j_{f,Q}(\Gamma) $ must consist exactly of two horizontal segments, both having one endpoint in $ Y_Q $. The length of each segment is less than $ r \beta^{-m+j}_0 < r $. Hence, both segments are contained in $ H_{Q}(r) $. By Proposition~\ref{pr:analytic}, $ \Phi^m_{f,Q}|_{H_Q(r)} $ is differentiable and $ \alpha(\Phi^m_{f,Q}) \ge \bar{\alpha} $. So $ \pi_{m}(\Gamma) $ consists of two segments, and $ a_{m}(\gamma) \ge \bar{\alpha} > 2 $ for every $ \gamma \in \pi_{m}(\Gamma) $. Therefore,
\[
\sup_{\Gamma \in \mathcal{H}(\tau)} \sum_{\gamma \in S_m(\Gamma)} \dfrac{1}{a_m(\gamma)} = \frac{2}{\bar{\alpha}} < 1. 
\]

The previous estimates imply the desidered property,
\[
\liminf_{\tau \to 0^+} \sup_{\Gamma \in \mathcal{H}(\tau)} \sum_{\gamma \in S_m(\Gamma)} \dfrac{1}{a_m(\gamma)} < \frac{2}{\bar{\alpha}}<1. 
\]
\end{proof}

\begin{remark}
\label{re:sataev}
Property~(A) is a version adapted to billiards of a condition introduced by Sataev~\cite{Sataev92} for general maps with singularities, which in turn is a stronger version of Condition~H4 in~\cite{Pesin92}. Condition~H4 is key in proving the existence of SRB measures for hyperbolic maps with singularities.
\end{remark}

In the next theorem, we prove the existence of SRB measures for $ \Phi_{f,Q} $, and several ergodic properties of such measures.

\begin{theorem}
\label{th:SRB}
Suppose that $ P \in \PP^*_n $. Then there exists $ \delta>0 $ such that for every $ f \in \RR^2 $ with $ \lambda(f)<\delta $ and every $ Q \in \PP_n $ with $ d(Q,P)<\delta $, the attractor $ \Lambda = \Lambda_{f,Q} $ is hyperbolic and regular, and there exist countably many ergodic SRB measures $ \mu_{1},\mu_2,\ldots $ of $ \Phi = \Phi_{f,Q} $ and countably many Borel subsets $ E_0,E_1,E_2,\ldots $ of $ \Lambda $ such that 
\begin{enumerate}
\item $ \Lambda = \bigcup_{i=0} E_i $ and $ E_i \cap E_j = \emptyset $ for all $ i \neq j $;
\item $ E_i \subset D $, $ \Phi(E_i)=E_i $, $ \mu_i(E_i)=1 $ 
and $ \Phi|_{E_i} $ is ergodic with respect to $ \mu_i $ for every $ i \ge 1 $;
\item for every $ i \ge 1 $, there exist $ k_i \in \Nn $ disjoint subsets $ B^{1}_{i},\ldots,B^{k_i}_{i} $ such that 
\begin{enumerate}
\item $ E_{i} = \bigcup^{k_i}_{j=1} B^{j}_{i} $; 
\item $ \Phi(B^{j}_{i}) = B^{j+1}_{i} $ for $ j = 1,\ldots,k_{i}-1 $, and $ \Phi(B^{k_i}_{i}) = B^{1}_{i} $;
\item $ \Phi^{k_i}|_{B^{j}_i} $ with the normalized restriction of $ \mu_i $ to $ B^{j}_i $ is a Bernoulli automorphism; 
\end{enumerate}
\item If $ \mu $ is an SRB measure of $ \Phi $, then there exist $ \alpha_{1},\alpha_2,\ldots $ with $ \sum_{i=1} \alpha_{i} = 1 $ such that $ \mu = \sum_{i} \alpha_{i} \mu_{i} $; 
\item if $ x \in D^{-}_{\epsilon} $ and $ \nu $ is a probability measure of $ M $ supported on $ W^{u}_{loc}(x) $ absolutely continuous with respect to the Riemannian volume of $ W^{u}_{loc}(x) $ and with density $ \kappa(x,\cdot) $ (see~\cite[Proposition~6]{Pesin92}), then every weak-* limit point of $
\mu_{n} = n^{-1} \sum^{n-1}_{k=0} \Phi^{k}_{*} \nu $ is an SRB measure of $ \Phi $; 
\item the set of periodic points of $ \Phi $ is dense in $ \Lambda $; 
\item for every $ i \ge 1 $, there exist $ C>0 $, $ \alpha>0 $ and $ r_0>0 $ such that $ \mu_i(N^{+}_{f,Q}(r)) \le C r^\alpha $ for every $ 0 < r < r_0 $.
\end{enumerate}
\end{theorem}

\begin{proof}
The theorem follows from results by Pesin on the existence and properties of SRB measures for general hyperbolic piecewise smooth maps. More precisely, conclusions~(1)-(4) follow from~\cite[Theorem~4]{Pesin92}, conclusion~(5) follows from \cite[Theorem~1]{Pesin92}, conclusion~(6) follows from \cite[Theorem~11]{Pesin92}, and conclusion~(7) follows from~\cite[Proposition~12]{Pesin92}. See also~\cite{Sataev92}, where Sataev obtained results that are stronger than those of Pesin (but under stronger hypotheses). To justify our claim, we need to show that the map $ \Phi_{f,Q} $ satisfies the hypothesis of Pesin's paper: the conditions called H1-H4 and the condition that $ \Lambda_{f,Q} $ is hyperbolic.

The map $ \Phi_{f,Q} $ satisfies conditions~H1 and~H2, because so does the billiard map $ \Phi_P $ with the specular reflection law (see~\cite[Theorem~7.2]{KS86}), and $ f $ and $ f^{-1} $ have bounded second derivates since $ f \in \mathcal{R}^2 $.
Since $ P \in \PP^*_n $, Lemma~\ref{le:expansion} and Proposition~\ref{pr:h} implies, respectively, the hyperbolicity of $ \Lambda_{f,Q} $ and Properties~(A) and (B) for every $ f \in \mathcal{R}^1 $ and every $ Q \in \PP_n $ sufficiently close to $ P $. Finally, Property~(A) implies~H3 (which is just the regularity of $ \Lambda_{f,Q} $) by~\cite[Proposition~3]{Pesin92}, and~(B) implies~H4.
\end{proof}

We call the sets $ E_1,\ldots,E_m $ the \emph{ergodic components} of $ \Phi $, and we call the sets $ B^1_i,\ldots,B^{k_i}_i $ the \emph{Bernoulli components} of $ E_i $.

\begin{remark}
Under the extra hypothesis that $ f'(\theta)>0 $ for every $\theta\in (-\pi/2,\pi/2) $, the previous theorem follows from a general result on polygonal billiards with contracting reflection laws~\cite[Theorem~4.12]{lxds}. 
As proved in Theorem~\ref{th:SRB}, that condition can be dropped when $ \lambda(f) $ is sufficiently small. 
\end{remark}

%

\subsection{Continuation of periodic points}
Throughout this section, $ P $ and $ Q $ are assumed to be polygons in $ \PP_n $ without parallel sides facing each other, and $ f $ is assumed to be a reflection law in $ \mathcal{R} $. Note that for such an $ f $, the map $ \Phi_{f,Q} $ is not necessarily invertible. Since $ P $ does not have parallel sides facing each other, $ \psi_P $ is piecewise expanding.

We prove in Theorem~\ref{th:fp} that each periodic point of the slap map $ \psi_P $ 
admits a continuation to a hyperbolic periodic point of the billiard map $ \Phi_{f,Q} $ provided that $ \lambda(f) $ and $ d(Q,P) $ are sufficiently small. 

Denote by $ B(x,r) $ the open ball of $ S^{1} \times (-\pi/2,\pi/2) $ centered at $ x $ of radius $ r>0 $. Also, given $ x \in D^{+}_{f,Q,\epsilon} $, we call the two curves contained in $ W^{s}_{loc}(x) $ having as endpoints $ x $ and a point of $ \partial W^{s}_{loc}(x) $ the \emph{components} of $ W^{s}_{loc}(x) $. 

\begin{theorem}
\label{th:fp}
Let $ s $ be a periodic point of $ \psi_{P} $ of period $ m \in \Nn $ whose orbit does not visit $ V_P $. Then there exist positive constants $ \delta, r $ and $ \ell $
such that
if $ \lambda(f)<\delta $ and $ d(Q,P)<\delta $, then 
\begin{enumerate}
\item $ \Phi_{f,Q} $ has exactly one hyperbolic periodic point $ x_{f,Q} $ of period $ m $ in $ B((s,0),r) $ converging to $ (s,0) $ as $ \lambda(f) + d(Q,P) \to 0 $,
\item the slope of $ E^s(x_{f,Q}) $ is smaller than $- 1/(2t_P(s,0))$,
\item if $ \gamma $ is a component of $ W^s_{loc}(x_{f,Q})) $, then $ |\gamma| \ge \ell $.
\end{enumerate}
\end{theorem}

Note that $ x := (s,0) $ is a hyperbolic periodic point of $ \Phi_{0,P} $ of period $ m $. 

\begin{defn}
\label{de:continuation}
We call $ x_{f,Q} $ the \emph{continuation} of $ s $ (or $ x = (s,0) $).	
\end{defn}  


To prove~Theorem~\ref{th:fp}, we need Lemma~\ref{le:fp}. Let $ d_1 $ be the $ C^1 $-distance between maps.

\begin{lemma}
\label{le:fp}	
Let $ s \in S^1 $, and suppose that there exists $ m \in \Nn
$ such that $ \psi^i_P(s) \notin V_P $ for every $ 0 \le i \le m-1 $. There are $ \delta_0>0 $ and $ r_0>0 $ such that if $ \lambda(f)<\delta_0 $, $ d(P,Q)<\delta_0 $ and $ x = (s,0) $, then the restrictions $ \Phi^m_{0,P}|_{B(x,r_0)} $ and $ \Phi^m_{f,Q}|_{B(x,r_0)} $ are both differentiable.
\end{lemma}

\begin{proof}
By the hypothesis on $ s $, there exists $ r_0>0 $ such that the restriction $ \Phi^m_{0,P}|_{B(x,r_0)} $ is differentiable. This implies that the components of the map $ \Phi^m_{f,Q} $	and the entries of the matrix $ d_x \Phi^m_{f,Q} $ are continuous in the variables $ f $ and $ Q $ at $ f=0 $ and $ Q=P $. Hence, there exists $ \delta_0>0 $ such that if $ \lambda(f)<\delta_0 $ and $ d(Q,P)<\delta_0 $, then $ \Phi^n_{f,Q}|_{B(x,r_0)} $ is differentiable as well. 
\end{proof}

\begin{remark}
By the previous lemma, we can think of $ \Phi^n_{f,Q}|_{B(x,r_0)} $ as a perturbation of $ \Phi^m_{0,P}|_{B(x,r_0)} $ (the unperturbed map). Then, Theorem~\ref{th:fp} turns out to be a corollary of a general theorem on the persistence of hyperbolic periodic points of smooth maps (without singularities) under small perturbations. See, for instance, Theorem~2.6 in the book~\cite{Wen}. However, to use this in our setting, we need a remark. The map $ \Phi^m_{0,P}|_{B(x,r_0)} $ is an endomorphism, whereas the aforementioned theorem assumes that the unperturbed map is a $ C^1 $ diffeomorphism. Nevertheless, it is clear from its proof, that the theorem continues to hold true even when the unperturbed map is just a $ C^1 $ endomorphism. This is so, because even though Lemma~2.5 of~\cite{Wen} -- the key step in the proof of Theorem~2.6 -- is formulated for hyperbolic linear isomorphisms, it is actually valid for hyperbolic linear endomorphisms (see also~\cite[Section~2.1]{Yo}). 
\end{remark}


\begin{proof}[Proof of Theorem~\ref{th:fp}]
Since the orbit of $ s $ does not visit $ V_P $, the $ \Phi_{0,P} $-orbit of $ x = (s,0) $ is defined, and $ x = (s,0) $ is a fixed point of $ \Phi^m_{0,P} $. Moreover, since $ P $ does not have parallel side facing each other, $ x $ is a hyperbolic periodic point. We can then apply Lemma~\ref{le:fp} to $ \Phi^m_{0,P} $ and $ x $. Let $ r_0>0 $ and $ \delta_0>0 $ be as in the lemma. Next, we apply~\cite[Theorem~2.6]{Wen} to $ \Phi^m_{0,P}|_{B(x,r_0)} $ and its perturbation $ \Phi^m_{f,Q}|_{B(x,r_0)} $ with $ \lambda(f)<\delta_0 $ and $ d(Q,P)<\delta_0 $. It follows that there exist $ 0 < \delta < \delta_0 $ and $ 0 < r < r_0 $ such that if $ \lambda(f)<\delta $ and $ d(Q,P)<\delta $, then $ \Phi^m_{f,Q}|_{B(x,r)} $ has a unique fixed point $ x_{f,Q} $ with the property that $ x_{f,Q} \to x $ as $ \lambda(f) + d(Q,P) \to 0 $. Since $ Q $ does not have sides facing each other, the map $ \Phi_{f,Q} $ is uniformly hyperbolic, implying that $ x_{f,Q} $ is hyperbolic. This proves conclusion~(1) of the theorem. 

Now, consider the fixed point $ x $ of $ \Phi^m_{0,P} $. Note that the $ (2,2) $-entry of $ d_x \Phi^m_{0,P} $ is equal to 0, since $ f = 0 $ in this case. Then, using~\eqref{eq:differential}, one can easily show that the slope of the stable direction of $ x $ is equal to $ -1/t_{P}(x) $. Since the entries of the matrix $ d_{x_{f,Q}} \Phi^m_{f,Q} $ are continuous functions of $ (f,Q) $ at $ f=0 $ and $ Q=P $, so is the slope of the stable direction of $ x_{f,Q} $. Thus, by further shrinking $ \delta $ if necessary, we obtain conclusion~(2) of the theorem with the lower bound for the slope equal to $ - 1/(2t_{P}(x)) $.

Now, we assume that $ \lambda(f)<\delta $ and $ d(Q,P)<\delta $. Since $ x_{f,Q} $ is a periodic point, the distance of its orbit from the singular set $ N^+_{f,Q} $ is bounded away from zero uniformly in $ f $ and $ Q $ (chosen as above). This implies that $ x_{f,Q} \in D^{+}_{f,Q,\epsilon,l} $ for some $ l \in \Nn $ uniformly in $ (f,Q) $. By Pesin's theory, there exists $ \ell>0 $ such that the length of each component of $ W^s_{loc}(x) $ is greater than or equal to $ \ell $ for every $ x \in D^{+}_{f,Q,\epsilon,l} $. This implies conclusion~(3) of the theorem.  
\end{proof}

\subsection{A criterion for ergodicity}
In this section, we will assume implicitly that $P$ is a polygon without parallel sides facing each other and that $ f \in \mathcal{R}^2 $. Hence, $ \Lambda_{f,P} $ is hyperbolic by Proposition~\ref{pr:hyperbolicity}. In particular, every periodic point of $ \Phi_{f,P} $ is hyperbolic. 

Recall that $ D^{\pm}_{f,P,\epsilon} $ and $ D^{\pm}_{f,P,\epsilon,l} $ are the Pesin sets of $ \Phi_{f,P} $ introduced in Subsection~\ref{su:pesin}. Given a $ C^1 $-curve $ W \subset K $, recall that $ \Vol_{W} $ denotes the normalized length of $ W $, i.e., the normalized volume of $ W $ induced by the metric $ d_{M} $. 

%

The next lemma plays a crucial role in the proof of Theorem~\ref{th:erg-periodic}. It tells us about the points contained in a given horizontal segment where local stable manifolds exist (c.f.~\cite[Proposition~3.4]{Sataev92} and~\cite[Lemma~1]{Pesin92}). The lemma is a consequence of Property~(B) in Proposition~\ref{pr:h}. 

\begin{lemma}
\label{le:dplus}
Let $ \Gamma $ be a horizontal segment. For every $ 0 < \tau < 1 $, there exists $ l_+ \in \Nn $ such that 
\[
\Vol_{\Gamma}(\Gamma \cap D^{+}_{f,P,\epsilon,l}) \ge 1-\tau \qquad \text{for all } l \ge l_+.
\] 	
\end{lemma}

\begin{proof}
Note that if $ x \in \Gamma \setminus D^+_{f,P,\epsilon,l} $, then $ d_{M}(\Phi^n_{f,P}(x),N^+_{f,P})<l^{-1} e^{-\epsilon n} $ for some $ n \ge 0 $. Thus, \[ 
\Gamma \setminus D^+_{f,P,\epsilon,l} \subset \bigcup^{\infty}_{n=0} \Gamma \cap \Phi^{-n}_{f,P}((N^+_{f,P})(l^{-1} e^{-\epsilon n})).
\]
By Property~(B) of Proposition~\ref{pr:h}, we obtain 
\begin{align*}
\Vol_{\Gamma}(\Gamma \setminus D^{+}_{f,P,\epsilon,l}) & \le \sum^{\infty}_{n=0} \Vol_{\Gamma}(\Gamma \cap \Phi^{-n}_{f,P}((N^+_{f,P})(l^{-1} e^{-\epsilon n})) \\
& \le \frac{B}{l} \sum^{\infty}_{n=0} e^{-\epsilon n} = \frac{B}{l(1-e^{-\epsilon})}.
\end{align*}
Hence,
\begin{align*}
\Vol_{\Gamma}(\Gamma \cap D^+_{f,P,\epsilon,l}) & = 1 - \Vol_{\Gamma}(\Gamma \setminus D^{+}_{f,P,\epsilon,l}) \\
& \ge 1 - \frac{B}{l(1-e^{-\epsilon})},
\end{align*}
which yields the wanted conclusion.
\end{proof}

Denote by $ C(\Lambda_{f,P}) $ be set of all continuous functions on the attractor $ \Lambda_{f,P} $. For every $ \varphi \in C(\Lambda_{f,P}) $, let
\[
\varphi^{+}(x) = \lim_{n \to +\infty} \frac{1}{n} \sum^{n-1}_{k=0} \varphi(\Phi^{k}_{f,P}(x))
\]
be the forward Birkhoff average of $ \varphi $. Also, let $ \mu_1,\mu_2,\ldots $ be the ergodic SRB measures of $ \Phi_{f,P} $, and let $ E_1,E_2,\ldots $ be the corresponding sets as in Theorem~\ref{th:SRB}. 

\begin{defn}
\label{de:generic}
For every $ i $, define 
\[ 
\mu_i(\varphi) = \int_{\Lambda} \varphi(x) \, d\mu_i(x) \qquad  \forall \varphi \in C(\Lambda_{f,P}),
\]
and 
\[
R_i = \left\{x \in E_i \colon \varphi^{+}(x) = \mu_i(\varphi) \quad \forall \varphi \in C(\Lambda_{f,P}) \right\}.
\]
\end{defn}

Since the sets $ E_i $'s are $ \Phi_{f,P} $-invariant and pairwise disjoint, so are the sets $ R_i $'s. Moreover, the separability of $ C(\Lambda_{f,P}) $ and the Birkhoff Ergodic Theorem imply that $ \mu_i(R_i)=\mu_i(E_i)=1 $ for every $ i $.

\begin{defn}
\label{de:typical}
For every $ i $, define $ \Delta_i $ to be the set of all 
$ x \in D^-_{f,P,\epsilon} $ for which there exists an open disk $ V_x $ in $ W^u_{loc}(x) $ containing $ x $ such that $ \Vol_{V_x}(V_x \cap R_i) = 1 $. 
\end{defn}

\begin{remark}
\label{re:full}
By the property of the conditional measures of an SRB measure, it follows that $ \mu_i(\Delta_i)=1 $ (see the definition of a $ u $-measure and the paragraph before~\cite[Proposition~9]{Pesin92}).
\end{remark}

The next results play a central role in the proofs of Theorem~\ref{thm:main2}, permiting to characterize the sets $ E_i $'s and the number of their Bernoulli components using the periodic points of the map $ \Phi $. We remark that our Theorem~\ref{th:erg-periodic} is similar to~\cite[Theorem~5.1]{Sataev99}.

%

\begin{proposition}
\label{pr:erg-periodic}
Let $ x \in \Delta_i $, and suppose that there are a periodic point $ x_0 $ of $ \Phi_{f,P} $ and an integer $ n \ge 0 $ such that 
\[ 
W^{s}_{loc}(x_{0}) \cap \Phi^{n}_{f,P}(V_{x}) \neq \emptyset,
\] 
where $ V_{x} $ is as in the definition of $ R_i $. Then there exists an open disk $ W $ of $ x_0 $ in $ W^{s}_{loc}(x_0) $ and a set $ W' \subset W \cap R_i $ such that $ \Vol_{W}(W')=1 $.
\end{proposition}

\begin{proof}
By hypothesis, there exists $ y \in W^{s}_{loc}(x_0) \cap \Phi^{n}_{f,P}(V_{x}) $. Let $ p $ be the period of $ x_0 $, and define $ y_k = \Phi^{kp}_{f,P}(y) $ for all $ k \ge 0 $. Clearly, we have $ y_k \in W^{s}_{loc}(x_0) $ and $ \lim_{k \to +\infty} y_k = x_0 $. 

Since the attractor $ \Lambda_{f,P} $ is hyperbolic, there exists $ k_0 \ge 0 $ such that
\[ 
W^u_{loc}(y_k) \cap \Phi^{n_1+kp}_{f,P}(V_{x}) = W^u_{loc}(y_k) \qquad \text{for all } k \ge k_0.
\]
Let $ k \ge k_0 $. We have $ V_k:=\Phi^{-(n_1+kp)}_{f,P}(W^u_{loc}(y_k)) \subset V_{x} $. Moreover, since the probability measure $ (\Phi^{n_1+kp}_{f,P})_*(\Vol_{V_k}) $ is equivalent to $ \Vol_{W^{u}_{loc}(y_k)} $, and the set $ R_i $ is $ \Phi $-invariant, we have
\[
\Vol_{W^u_{loc}(y_k)}(W^u_{loc}(y_k) \cap R_i)=1.
\]

The periodic point $ x_0 $ is hyperbolic, and so $ x_0 \in D^{-}_{f,P,\epsilon,l_-} $ for some $ l_- \in \Nn $. It follows that there is $ k_1 \ge k_0 $ such that 
\begin{equation}
\label{eq:size}	
y_k \in D^-_{f,P,\epsilon,2l_-} \qquad \text{for all } k \ge k_1. 
\end{equation}
This implies, in particular, that the size of $ W^{u}_{loc}(y_k) $ is uniformly bounded from below by some positive constant depending on $ \epsilon $ and $ l_- $. Since $ D^-_{f,P,\epsilon,2l_-} $ is closed, by~\cite[Propositions~1 and~4]{Pesin92}, the sequence of curves $ W^{u}_{loc}(y_k) $ converges in the $ C^1 $-topology to an open disk $ W \subset W^{u}_{loc}(x_0) $ containing $ x_0 $. 

By Lemma~\ref{le:dplus}, for every $ 0 < \delta < 1 $, there exists $ l_+ \ge 0 $ such that
\[
\Vol_{W}(W \cap D^{+}_{f,P,\epsilon,l_+}) \ge 1-\delta.
\] 
This together with~\eqref{eq:size} implies that there exist an open disk $ W_1 \subset W $ containing $ x_0 $ and an integer $ k_2 \ge k_1 $ such that 
\[
\Vol_{W}(W_1 \cap D^+_{f,P,\epsilon,l_+}) \ge 1 - \delta/2,
\] 
and
\[
W^s_{loc}(w) \cap W^u_{loc}(y_{k_2}) \neq \emptyset \qquad \text{for all } w \in W_1 \cap  D^+_{f,P,\epsilon,l_+}.
\]

Let 
\[
W_2 = \left\{w \in W_1 \cap  D^+_{f,P,\epsilon,l_+} \colon W^s_{loc}(w) \cap W^u_{loc}(y_{k_2}) \cap R_i \neq \emptyset \right\}.
\]
It is a well known fact that if $ \varphi \in C(\Lambda) $, $ w \in D^+_{f,P,\epsilon} $ and $ \varphi^{+}(w) $ exists, then $ \varphi^{+}(z) = \varphi^{+}(w) $ for every $ z \in W^s_{loc}(x) $. Therefore, if $ w \in W_2 $, then $ W^s_{loc}(w) \cap R_i \neq \emptyset $, and so $ \varphi^{+}(w)=\mu_i(\varphi) $ for every $ \varphi \in C(\Lambda_{f,P}) $. By the absolute continuity of the stable foliation (see~\cite[Proposition~10]{Pesin92}),
\[
\Vol_{W}(W_2) = \Vol_{W}(W_1 \cap D^+_{f,P,\epsilon,l_+}) \ge 1-\delta/2.
\]

We have just proved that for every $ 0 < \delta < 1 $, there exists a subset $ W_2=W_2(\delta) \subset W $ such that $ \Vol_{W}(W_2) \ge 1 - \delta/2 $, and $ \varphi^+|_{W_2} = \mu_i(\varphi) $ for every $ \varphi \in C(\Lambda_{f,P}) $. From this, it is straightforward to obtain the existence of a set $ W' \subset W $ such that $ \Vol_{W}(W')=1 $ and $ \varphi^+|_{W'} = \mu_i(\varphi) $ for every $ \varphi \in C(\Lambda_{f,P}) $. 
\end{proof}

\begin{theorem}
\label{th:erg-periodic}
Let $ x_1 \in \Delta_i $ and $ x_2 \in \Delta_j $, and suppose that there are a periodic point $ x_0 $ of $ \Phi_{f,P} $ and two integers $ n_1 \ge 0 $ and $ n_2 \ge 0$ such that $ W^{s}_{loc}(x_{0}) \cap \Phi^{n_{1}}_{f,P}(V_{x_1}) \neq \emptyset $ and $ W^{s}_{loc}(x_{0}) \cap \Phi^{n_{2}}_{f,P}(V_{x_2}) \neq \emptyset $, where $ V_{x_1} $ and $ V_{x_2} $ are the sets corresponding to $ x_1 $ and $ x_2 $ as in the definition of $ R_i $. Then $ i=j $. 
\end{theorem}


\begin{proof}
Both $ x_1 $ and $ x_2 $ satisfy the hypotheses of Proposition~\ref{pr:erg-periodic}. Thus, there exist two open disks $ W_1 $ and $ W_2 $ of $ x_0 $, and two sets $ W'_1 \subset W_1 \cap R_i $ and $ W'_2 \subset W_2 \cap R_j $ such that $ \Vol_{W_1}(W'_1)=\Vol_{W_2}(W'_2)=1 $. It follows that $ R_i \cap R_j \neq \emptyset $. Hence $ R_i = R_j $, i.e., $ i = j $.
\end{proof}

We now prove a proposition that allows us to estimate the number of Bernoulli components of an ergodic SRB measure. 

Let $ \mu = \mu_i $ be one of the ergodic SRB measures of $ \Phi =\Phi_{f,P} $, and let $ E = E_i $ be the corresponding ergodic component. Suppose that $ \mu $ has $ n $ Bernoulli components. Then, it follows from Theorem~\ref{th:SRB} that there exist an integer $ n > 0 $ and a probability measure $ \mu' $ such that $ \Phi^n $ endowed with $ \mu' $ is a Bernoulli automorphism, and $ \mu $ is the arithmetic average of $ \mu',\Phi_{*} \mu',\ldots,\Phi^{n-1}_{*} \mu' $.
%
Define the sets $ R' $ and $ \Delta' $ for the measure $ \mu' $ and the map $ \Phi^n $ exactly as the sets $ R_i $ and $ \Delta_i $ for the measure $ \mu_i $ and the map $ \Phi $ in Definitions~\ref{de:generic} and \ref{de:typical}. Remark~\ref{re:full} applies to $ \Delta' $ and $ \mu' $ as well, and so $ \mu'(\Delta')=1 $. Also, given $ x \in \Delta' $, let $ V_x $ be the open disk in $ W^u_{loc}(x) $ containing $ x $ as in Definition~\ref{de:typical}.

\begin{proposition}
\label{pr:period}
Let $ x \in \Delta' $, and suppose that there exist an integer $ j  \ge 0 $ and a periodic point $ x_0 $ of $ \Phi $ of period $ p $ such that
\[
W^s_{loc}(x_0) \cap \Phi^{j}(V_x) \neq \emptyset.
\]
Then $ n $ is a divisor of $ p $. 
\end{proposition}

\begin{proof}
By Proposition~\ref{pr:erg-periodic}, there exists a neighborhood $ W \subset W^u_{loc}(x_0) $ of $ x_0 $ and a set $ W' \subset W \cap R' $ with $ \Vol_{W}(W')=1 $. Since $ x_0 $ is periodic and $ W' \subset W^u_{loc}(x_0) $, we have $ \Phi^{-p}(W') \subset W' $. Moreover, $ \Vol_{W}(\Phi^{-p}(W'))>0 $ because the measure $ \Phi^{p}_* \Vol_{W} $ is equivalent to $ \Vol_{\Phi^{-p}(W)} $.
This yields 
\[
W' \cap \Phi^{-p}(W') \neq \emptyset. 
\]
Since $ W' \subset R' $, the set $ W' $ is contained in a Bernoulli component of $ \mu $, and so $ n $ must be a divisor of $ p $.
\end{proof}

\begin{remark}
The proofs of the propositions in this subsection do not use any specific property of $ \Phi_{f,P} $ or of its local invariant manifolds. For this reason, although the propositions were proved only for the billiard map $ \Phi_{f,P} $, they apply to general hyperbolic map with singularities. 
\end{remark}

\section{Proof of Theorem~\ref{thm:main2}}
\label{se:bijection theta}

In this section, we prove Theorems~\ref{th:main-partone}, \ref{th:bernoulli} and~\ref{th:basin} and Corollary~\ref{co:main-partone}. Together, they form Theorem~\ref{thm:main2} (the same as Theorem~\ref{th:main}). 

Throughout this section, we assume that $ P \in \PP^*_n $. Recall that $ \SRB(\psi_P) $ and $ \SRB(\Phi_{f,Q}) $ denote the set of ergodic acips of $ \psi_P $ and the set of the ergodic SRB measures of $ \Phi_{f,Q} $, respectively.

\subsection{The set $ \FF(\nu) $}
Consider $ \nu \in \SRB{\psi_P} $. Let $ \eta>0 $ be as in Lemma~\ref{le:growth}. From Part~(3) of Theorem~\ref{th:slap acip} and Condition~(*), it follows that there exists a finite subset $ \FF(\nu) $ of $\operatorname{int}(\supp \nu) $ with the following properties:
\begin{enumerate}
\item $ \FF(\nu) $ consists of periodic points of $ \psi_P $ whose orbits do not visit any vertex of $ P $,
\item $ \FF(\nu) $ is $ \eta/6 $-dense in $ \supp \nu $,
\item for every $ s \in \FF(\nu) $, there exists $ z \in (s-\eta/9,s+\eta/9) \cap \FF(\nu) $ with $ z \neq s $ such that the great common divisor of the periods of $ s $ and $ z $ equals the number of the exact components of $ \nu $.	
\end{enumerate}

Given $ s \in S^1 $ and $ r>0 $, let $ B(s,r) = (s-r,s+r) $ be the open interval of $ S^1 $ centered at $ s $ of radius $ r $.

\begin{lemma}
\label{le:uno}	
For any pair $ s_1,s_2 \in \FF(\nu) $, there exist an open interval $ I \subset B(s_1,\eta/3) \cap \supp \nu $, integers $ k,m \ge 0 $ and $ s_0 \in \FF(\nu) $ such that 
\begin{enumerate}
\item $ \psi^{k+m}_{P}|_{I} $ is differentiable (so $ \psi^{k}_{P}(I) $ and $ \psi^{k+m}_{P}(I) $ are intervals), 
\item $ \psi^{k}_{P}(I) \subset B(s_2,\eta/3) \cap \supp \nu $,
\item $ |\psi^{k+m}_{P}(I)|>\eta $,
\item $ B(s_0,\eta/3) \subset \psi^{k+m}_{P}(I) \subset \supp \nu $.
\end{enumerate}
\end{lemma}

\begin{proof}
Let $ B_i = B(s_i,\eta/3) $ for $ i=1,2 $. From $ s_i \in \operatorname{int}(\supp \nu) $, it follows that $ B_i \cap \supp \nu $ is an interval, and so $ \nu(B_i)>0 $. Since $ \nu $ is ergodic, there exist $ s \in B_1 \cap \supp \nu $ and an integer $ k \ge 0 $ such that $ \psi^k_P(s) \in B_2 $.
The set of points whose forward orbit meets a vertex of $ P $ has zero $ \nu $-measure. Then, we can assume without loss of generality that $ \psi^i_P(s) \notin V_P $ for every $ 0 \le i \le k $. Thus,  there exists a subinterval $ I_1 $ of $ B_1 \cap \supp \nu $ with $ s \in \operatorname{int}(I_1) $ such that $ \psi^k_P(I_1) $ is differentiable. In particular, $ \psi^k_P(I_1) $ is a subinterval of $ B_2 \cap \supp \nu $. 

By Lemma~\ref{le:growth}, there are an integer $ m \ge 0 $ and an open interval $ J \subset \psi^k_P(I_1) $ such that $ {\psi^m_P}|_J $ is differentiable, and $ \psi^m_P(J) $ is an interval contained in $ \supp \nu $ with $ |\psi^m_P(J)|>\eta $. We conclude that there exists an open interval $ I \subset I_1 $ with $ \psi^k_P(I)=J $ such that $ {\psi^{k+m}_P}_I $ is differentiable, $ \psi^{k+m}_P(I) \subset B_2 \cap \supp \nu $, $ \psi^{k+m}_P(I) \subset \supp \nu $ and $ |\psi^{k+m}_P(I)|>\eta $. Finally, since $ \FF(\nu) $ is $ \eta/6 $-dense in $ \supp \nu $, there exists $ s_0 \in \FF(\nu) $ such that $ B(s_0,\eta/3) \subset \psi^{k+m}_P(I) $.
\end{proof}


\begin{lemma}
\label{le:unomezzo}
There exists $ \delta_1>0 $ such that for any $ s_1,s_2 \in \FF(\nu) $, there are integers $ m_1,m_2 \ge 0 $ and $ s_0 \in \FF(\nu) $ for which if $ \lambda(f)<\delta_1 $, $ d(Q,P)<\delta_1 $ and $ \Gamma_1 $ and $ \Gamma_2 $ are horizontal segments satisfying $ B(s_i,\eta/3) \subset \pi_{s} \left(\Gamma_i\right) $ for $ i=1,2 $, then
\[
B(s_0,\eta/3) \subset \pi_{s} \left(\Phi^{m_i}_{f,Q}(\Gamma_i)\right), \qquad i=1,2.
\]
\end{lemma}

\begin{proof}
For any pair of points $ s_1,s_2 \in \FF(\nu) $, denote bs $ I(s_1,s_2) $, $ s_0(s_1,s_2) $, $ m(s_1,s_2) $ and $ k(s_1,s_2) $ the interval, the point of $ \FF(\nu) $ and the two positive integers 
as in Lemma~\ref{le:uno}. Next, define 
\[ 
m_1(s_1,s_2) = m(s_1,s_2)+k(s_1,s_2) \quad \text{and} \quad m_2(s_1,s_2)=k(s_1,s_2).
\]
Since $ \FF(\nu) $ is finite, $ m_1 $ and $ m_2 $ are bounded functions on $ \FF(\nu) \times \FF(\nu) $. For this reason, the assumption on $ \Gamma_1 $ and $ \Gamma_2 $ and Lemma~\ref{le:uno},
we can find a $ \delta_1>0 $ such that if $ \lambda(f)<\delta_1 $ and $ d(Q,P)<\delta_1 $, then for any pair $ s_1,s_2 \in \FF(\nu) $, the sets $ \Phi^{m_1(s_1,s_2)}_{f,Q}(\Gamma_1) $ and $ \Phi^{m_2(s_1,s_2)}_{f,Q}(\Gamma_2) $ will be so close to the intervals $ \psi^{m_1(s_1,s_2)}_{P}(I(s_1,s_2)) $ and $ \psi^{m_2(s_1,s_2)}_{P}(I(s_1,s_2)) $ that
\[
B\left(s_0(s_1,s_2),\eta/3)\right) \subset \pi_{s} \left(\Phi^{m_i(s_1,s_2)}_{f,Q}(\Gamma_i)\right), \qquad i=1,2.
\]	
\end{proof}

\subsection{The set $ \FF_{f,Q}(\nu) $}
\label{su:F}
For every $ s \in \FF(\nu) $, denote by $ \delta(s) $, $ \kappa(s)=2t_P(s,0) $ and $ \ell(s) $ the constants in Theorem~\ref{th:fp}. Next, define $ \bar{\delta} = \min_{s \in \FF(\nu)} \delta(s) $, $ \bar{\kappa} = \min_{s \in \FF(\nu)} \kappa(s) $ and $ \bar{\ell} = \min_{s \in \FF(\nu)} \ell(s) $.
 If $ \lambda(f)<\min\{\delta_1,\bar{\delta}\} $ and $ d(Q,P)<\min\{\delta_1,\bar{\delta}\} $, then Theorem~\ref{th:fp} says that there exists a continuation $ x_{f,Q} = (s_{f,Q},\theta_{f,Q}) $ for every $ (s,0) $ with $ s \in \FF(\nu) $. 

\begin{lemma}
\label{le:distance}
There is $ 0<\delta_2<\min\{\delta_1,\bar{\delta}\} $ such that if $ \lambda(f)<\delta_2 $ and $ d(Q,P)<\delta_2 $, then for every $ s \in \FF(\nu) $, the local stable manifold $ W^s_{loc}(x_{f,Q}) $ of the continuation of $ (s,0) $ intersects both lines $ \theta = - \lambda(f) \pi/2 $ and $ \theta = - \lambda(f) \pi/2 $ at points with $ s $-coordinate contained in the interval $ (s-2\eta/9,s+2\eta/9) $.	
\end{lemma}

\begin{proof}
Suppose that $ \lambda(f)<\min\{\delta_1,\bar{\delta}\} $ and $ d(Q,P)<\min\{\delta_1,\bar{\delta}\} $. A straightforward computation -- which we omit -- shows that if we also require 
\[
\lambda(f) < \min \left \{\dfrac{\bar{\ell}}{\pi \sqrt{1+\bar{\kappa}^2}},\dfrac{\eta}{9 \pi} \right\},
\]
then for every $ s \in \FF(\nu) $, the local stable manifold of the continuation $ x_{f,Q} $ of $ (s,0) $ intersects both lines $ \theta = - \lambda(f) \pi/2 $ and $ \theta = - \lambda(f) \pi/2 $ at points with $ s $-coordinate contained in $ (s_{f,Q}-\eta/9,s_{f,Q}+\eta/9) $. The existence of $ 0<\delta_2<\min\{\delta_1,\bar{\delta}\} $ with the wanted property follows from the fact that $ x_{f,Q} \to (s,0) $ as $ \lambda(f)+d(P,Q) \to 0 $. 
\end{proof}

For $ \lambda(f)<\delta_2 $ and $ d(Q,P)<\delta_2 $, define
\[
\FF_{f,Q}(\nu) = \left\{x_{f,Q} \in M_Q \colon s \in \FF(\nu)\right\}.
\]
Fix $ 0 < \zeta < \eta/6 $. Proposition~\ref{pr:trapping3} implies that there exists $ 0<\delta_3 < \delta_2 $ such that if $ \lambda(f) < \delta_3 $ and $ d(Q,P)<\delta_3 $, then there are trapping regions $ U(\nu) $ and  $ W(\nu) $ for $ \psi_{Q} $ and $ \Phi_{f,Q} $, respectively.
Finally, note that $ \FF(\nu) $ is $ \eta/3 $-dense in $ U(\nu) $ because of our choice of $ \zeta $.

\subsection{The set $ \GG $}
Let $ \BB \subset S^1 $ be the union of the basins of the ergodic acips of $ \psi_P $. For every $ s \in S^1 $ and $ r>0 $, define 
\[ 
\Sigma(s,r) = (s-r,s+r) \times \left(-\pi r/2,\pi r/2 \right).
\] 

\begin{lemma}
\label{le:g}
There exist a $ \eta/2 $-dense finite set $ \GG \subset \BB $ in $ S^1 $ and $ \delta_4>0 $ such that if $ \lambda(f)<\delta_4 $ and $ d(Q,P)<\delta_4 $, then for every $ s \in \GG $, there are an ergodic acip $ \tilde{\nu} \in \SRB(\psi_P) $ and an integer $ k \ge 0 $ for which $ \Phi^k_{f,Q}(\Sigma(s,\delta_4)) \subset \operatorname{int}(W(\tilde{\nu})) $.
\end{lemma}

\begin{proof}
Let $ \mathcal{S} $ be the set of all $ s \in \BB $ such that $ \psi^{i}_{P}(s) $ is a vertex of $ P $ for some $ i \ge 0 $. The set $ \mathcal{S} $ is at most countable, and $ \BB $ has full Lebesgue measure by Part~(4) of Theorem~\ref{th:slap acip}. Hence, $ \BB \setminus \mathcal{S} $ has full Lebesgue measure as well, and so it contains a finite set $ \GG $ that is $ \eta/2 $-dense in $ S^1 $. 

Let $ s \in \GG $. Then, there exists $ \tilde{\nu} \in \SRB(\psi_P) $ such that $ s \in B(\tilde{\nu}) $. Since $ \supp \tilde{\nu} \subset \operatorname{int}(U(\tilde{\nu})) $ (see Proposition~\ref{pr:trapping3}), Urysohn's Lemma guarantees the existence of a continuos function $ \phi \colon I \to [0,1] $ such that $ \phi $ is identically equal to $ 1 $ on $ \supp \tilde{\nu} $ and is identically equal to $ 0 $ on the closure of the complement of $ S^1 \setminus U(\tilde{\nu}) $.
Since $ s \in B(\tilde{\nu}) $,  
\[
\lim_{k \to +\infty} \frac{1}{k} \sum^{k-1}_{i=0} \phi(\psi_P(s)) = \int_{S^1} \phi(s) d\tilde{\nu}(s) = 1.
\]
Hence, there exists $ k \ge 0 $ such that $ \psi^{k}_{P}(s) \in \operatorname{int}(U(\tilde{\nu})) $. Since $ s \notin \mathcal{S} $, it follows that $ \psi^i_P(s) \notin V_P $ for every $ i $, and so
\[ 
\Phi^{k}_{0,P}(s,0) = (\psi^{k}_P(s),0).
\] 
This and the fact that $ U(\tilde{\nu}) \times \{0\} = W(\tilde{\nu}) $
imply $ \Phi^{k}_{0,P}(s,0) \in \operatorname{int}(W(\tilde{\nu})) $. 

Finally, note that the map $ (f,Q,x) \mapsto \Phi^{k}_{f,Q}(x) $ is continuous at $ (0,P,(s,0)) $. Thus, there exits $ \delta_4 > 0 $ such that 
\[ 
\Phi^{k}_{f,Q}(\Sigma(s,\delta_4)) \subset \operatorname{int}(W(\tilde{\nu})) 
\]
provided that $ \lambda(f)<\delta_4 $ and $ d(Q,P)<\delta_4 $. 
\end{proof}

\subsection{Ergodic SRB measures} 
The constant $ \delta_3 $ derived in Subsection~\ref{su:F} depends on the ergodic acip $ \nu $. To emphasize this dependence, we write $ \delta_3(\nu) $ instead of $ \delta_3 $. Define $ \bar{\delta}_3 $ to be the minimum of $ \delta_3(\nu) $ over all $ \nu \in \SRB(\psi_P) $. 

Choose $ 0<\delta_5<\min \{\bar{\delta}_3,\delta_4\} $ so that Theorem~\ref{th:SRB} applies to each $ \Phi_{f,Q} $ with $ \lambda(f) < \delta_5 $ and $ d(Q,P)<\delta_5 $, and guarantees that the existence of ergodic SRB measures for $ \Phi_{f,Q} $. In the rest of the proof, we will implicitly assume that $ \lambda(f) < \delta_5 $ and $ d(Q,P)<\delta_5 $. 

\begin{defn} 
For every $ \nu \in \SRB{\psi_P} $, define 
\[ 
H_{f,Q}(\nu) = \left\{ \mu \in \SRB(\Phi_{f,Q}) \colon \mu(\supp \mu \cap W(\nu))=1 \right\}.
\] 
\end{defn}

Given $ \mu \in H_{f,Q}(\nu) $, denote by $ R(\mu) $ and $ \Delta(\mu) $ the sets corresponding to $ \mu $ as in Definitions~\ref{de:generic} and \ref{de:typical}, respectively. By Remark~\ref{re:full}, $ \mu\left(\Delta(\mu) \cap W(\nu)\right) = 1 $.

\begin{lemma}
\label{le:due}
Let $ \mu \in H_{f,Q}(\nu) $. If $ x \in \Delta(\mu) \cap W(\nu) $, then there exist $ j \ge 0 $ and $ s \in \FF(\nu) $ such that 
\[
B(s,\eta/3) \subset \pi_{s}(\Phi^j_{f,Q}(V_x)), 
\]	
where $ V_x \subset W^u_{loc}(x) $ is the set associated to $ x $ as in Definition~\ref{de:typical}.
\end{lemma}

\begin{proof}
Since $ V_x $ is a horizontal segment, Lemma~\ref{le:growth} implies that there exists $ j \ge 0 $ such that $ \Phi^j_{f,Q}(V_x) $ contains a horizontal segment $ \Gamma $ with $ |\Gamma|>\eta $. Since $ \FF(\nu) $ is $ \eta/6 $-dense in $ U(\nu) $, and $ \pi_{s}(W(\nu)) = U(\nu) $, there is $ s \in \FF(\nu) $ such that 
\[
B(s,\eta/3) \subset \pi_{s}(\Gamma) \subset \pi_{s}(\Phi^j_{f,Q}(V_x)).
\]  
\end{proof}

\begin{lemma}
\label{le:tre}
$ \# H_{f,Q}(\nu) \le 1 $.
\end{lemma}

\begin{proof}
Let $ \mu_{n_1},\mu_{n_2} \in H_{f,Q}(\nu) $.
For $ i=1,2 $, 
pick $ x_i \in \Delta(\mu_{n_i}) \cap W(\nu) $. Such an $ x_i $ exists, because $ \mu_i(\Delta(\mu_{n_i}) \cap W(\nu))=1 $. By Lemma~\ref{le:due},  there exist $ j_1,j_2 \ge 0 $ and $ s_1,s_2 \in \FF(\nu) $ such that 
\begin{equation}
\label{eq:first}
B(s_i,\eta/3) \subset \pi_{s} \left(\Phi^{j_i}_{f,Q}(V_{x_i})\right), \qquad \text{for } i=1,2.
\end{equation}
It follows that there exist two horizontal segments $ \Gamma_1 \subset \Phi^{j_i}_{f,Q}(V_{x_1}) $ and $ \Gamma_2 \subset \Phi^{j_i}_{f,Q}(V_{x_2}) $ whose images under $ \pi_s $ contain the intervals $ B(s_1,\eta/3) $ and $ B(s_2,\eta/3) $, respectively.

By applying Lemma~\ref{le:unomezzo} to $ \Gamma_1 $ and $ \Gamma_2 $, we can conclude that there exist two integers $ m_1,m_2 \ge 0 $ and $ s_0 \in \FF(\nu) $ such that for each $ i=1,2 $,
\begin{equation}
\label{eq:three}
B(s_0,\eta/3) \subset \pi_{s} \left(\Phi^{j_i+m_i}_{f,Q}(V_{x_i})\right).	
\end{equation}

Let $ x_0 \in \FF_{f,Q}(\nu) $ be the periodic point of $ \Phi_{f,Q}$ corresponding to $ s_0 $. By Lemma~\ref{le:distance}, $ W^{s}_{loc}(x_0) $ intersects both lines $ \theta = \pm \lambda(f) \pi/2 $ at points with $ s $-coordinate contained in $ B(s_0,\eta/3) $. Since each $ \pi_{\theta} \left(\Phi^{j_i+m_i}_{f,Q}(V_{x_i})\right) $ is contained in the strip $ |\theta | < \lambda(f) \pi/2 $, it follows that for each $ i=1,2 $,
\[
W^s_{loc}(x_0) \cap \Phi^{j_i+m_i}_{f,Q}(V_{x_i}) \neq \emptyset.
\] 	
We can now apply Theorem~\ref{th:erg-periodic} to $ \mu_{n_1} $ and $ \mu_{n_2} $, and conclude that $ n_1 = n_2 $, i.e., $ \mu_{n_1}=\mu_{n_2} $.
\end{proof}

\begin{lemma}
\label{le:2}
$\# H_{f,Q}(\nu) =1$.
\end{lemma}

\begin{proof} 
By Lemma~\ref{le:tre}, it is enough to prove that $ H_{f,Q}(\nu) \neq \emptyset $. Let $ x \in \FF_{f,Q}(\nu) $. Since $ x $ is a hyperbolic periodic point, we have $ x \in D^{-}_{f,Q,\epsilon} $. Parts~(5) of Theorem~\ref{th:SRB} applied to $ x $ implies that $ \Phi_{f,Q} $ has an SRB measure $ \tilde{\mu} $. Since $ \FF_{f,Q}(\nu) \subset W(\nu) $, it follows from Proposition~\ref{pr:trapping3} that $ \supp \tilde{\mu} $ is contained in the closure of $ W(\nu) $. By the ergodic decomposition of $ \tilde{\mu} $ (see~Part~(4) of Theorem~\ref{th:SRB}), there exists $ \mu \in \SRB(\Phi_{f,Q}) $ such that   
\[ 
\supp \mu \subset \supp \tilde{\mu} \subset \overline{W(\nu)} = U(\nu) \times \left[-\frac{\pi}{2}\lambda(f),\frac{\pi}{2}\lambda(f)\right].
\] 
Now, by Part~(7) of Theorem~\ref{th:SRB}, $ \tilde{\mu}(\supp \mu \cap \partial K_{f,Q})=0 $. Since $ \partial K_{f,Q} \subset S^1 \times \{-\pi \lambda(f)/2,\pi \lambda(f)/2 \} $, we have $ \mu(\supp \mu \cap W(\nu))=1 $. We conclude that $ \mu \in H_{f,Q}(\nu) $.
\end{proof}

Lemmas~\ref{le:tre} and~\ref{le:2} allow us to define $ \Theta_{f,Q} \colon \SRB(\psi_{P}) \to \SRB(\Phi_{f,Q}) $ by $ \Theta_{f,Q}(\nu) = \mu $ with $\mu \in H_{f,Q}(\nu) $. Next, we prove that $\Theta_{f,Q} $ is a bijection. 

\begin{lemma}
$\Theta_{f,Q}$ is one-to-one.
\end{lemma}

\begin{proof}
Suppose that there exist two distinct mesures $\nu_1, \nu_2 \in \SRB(\psi_{P}) $ such that $ \mu =\Theta_{f,Q}(\nu_1)=\Theta_{f,Q}(\nu_2) $. Then $ \mu(W(\nu_1) \cap W(\nu_2))=1 $,
contradicting $ W(\nu_1) \cap W(\nu_2) = \emptyset $ (see Proposition~\ref{pr:trapping3}).
\end{proof}  

\begin{lemma}
\label{le:onto}
$\Theta_{f,Q}$ is onto.
\end{lemma}

\begin{proof} 
We prove that given $ \mu \in \SRB(\Phi_{f,Q}) $, there exists $ \nu \in \SRB(\psi_P) $ such that $ \mu \in H_{f,Q}(\nu) $, i.e., $ \mu(\supp \mu \cap W(\nu))=1 $.

Pick $ x \in \Delta(\mu) $, and let $ V_x $ be the open disk of $ W^u_{loc}(x) $ as in Definition~\ref{de:typical}. By Lemma~\ref{le:growth}, there exists an integer $ i > 0 $ such that  $\Phi_{f,Q}^i(V_x)$ contains a horizontal segment $ \Gamma_1 $ with $ |\Gamma_1|>\eta $. 

Let $ \GG $ be the $ \eta/2 $-dense set in $ S^1 $ as in Lemma~\ref{le:g}. Since $ \Gamma_1 $ is contained in $ S^1 \times (-\pi \lambda(f)/2,\pi \lambda(f)/2) $, we have $ \Gamma_1 \cap \Sigma(s,\delta_5) \neq \emptyset $ for some $ s \in \GG $. By Lemma~\ref{le:g}, there exist an integer $ k \ge 0 $, a measure $ \nu \in \SRB(\psi_P) $ and a horizontal segment $ \Gamma_2 \subset \Gamma_1 \cap \Sigma(s,\delta_5) $ such that $ \Gamma':=\Phi^{k}_{f,Q}(\Gamma_2) $ is a horizontal segment contained in $ W(\nu) $. 

Combining the previous conclusions, we obtain that there is a horizontal segment $ \Gamma_0 \subset V_x $ such that $ \Gamma'=\Phi^{k+i}_{f,Q}(\Gamma_0) \subset W(\nu) $. Now, recall that $ \Vol_{V_x}(V_x \cap R(\mu))=1 $, and that $ R(\mu) $ is $ \Phi_{f,Q} $-invariant. Since the push-forward of $ \Vol_{\Gamma_0} $ by $ \Phi^{k+i}_{f,Q} $ is equivalent to $ \Vol_{\Gamma'} $, it follows that $ \Vol_{\Gamma'}(\Gamma' \cap R(\mu))=1 $. But $ \Gamma' \subset W(\nu) $, and so $ R(\mu) \cap W(\nu) \neq \emptyset $. 

We claim that $ \mu(R(\mu) \cap W(\nu))=1 $. Indeed, let $ A = R(\mu) \setminus W(\nu) $, and suppose that $ \mu(A)>0 $. Let $ z \in A $ be a $ \mu $-point of density of $ A $. Since $ K_{f,Q} \setminus W(\nu) $ is open, there are a compact neighborhood $ C $ of $ z $ and an open neighborhood $ O $ of $ z $ such that $ C \subset O \subset K_{f,Q} \setminus W(\nu) $. By Urysohn's Lemma, there is a continuous $ \phi \colon K_{f,Q} \to [0,1] $ such that $ \phi|_{C} \equiv 1 $ and $ \phi|_{K_{f,Q} \setminus O} \equiv 0 $. Let $ x \in R(\mu) \cap W(\nu) $. Using the function $ \phi $ as in the proof of Lemma~\ref{le:g}, one can show that $ \Phi^j_{f,Q}(x) \in O $ for some $ j \in \Nn $. In particular, $ \Phi^{j}_{f,Q}(x) \notin W(\nu) $. However, that is impossible, since $ W(\nu) $ is $ \Phi_{f,Q} $-forward invariant by Proposition~\ref{pr:trapping3}, and so $ \Phi^j_{f,Q}(x) \in W(\nu) $. Therefore, $ \mu(A)=0 $, whic is equivalent to $ \mu(R(\mu) \cap W(\nu))=1 $. Since $ \mu(R(\mu) \cap \supp \mu)=1 $, we conclude that $ \mu(\supp \mu \cap W(\nu))=1 $. 
\end{proof}


The previous propositions prove the following. 

\begin{theorem}
\label{th:main-partone}
Let $ P \in \PP^*_n $. There is $ \delta_5>0 $ such that if $ \lambda(f)<\delta_5 $ and $ d(Q,P)<\delta_5 $, then there exists a bijection $ \Theta_{f,Q} \colon \SRB(\psi_P) \to  \SRB(\Phi_{f,Q}) $. Moreover, for every $ \nu \in \SRB(\psi_P) $, the support of $ \Theta_{f,Q}(\nu) $ is contained in the closure of the trapping set $ W(\nu) $.  
\end{theorem}

The next corollary is a direct consequence of the previous theorem and the fact that the number of ergodic acips of $ \psi_P $ is bounded from above by the cardinality of the singular set $ S_P $, which is not larger than $ n $, the number of sides of $ P $.

\begin{corollary}
\label{co:main-partone}
Under the hypotheses of Theorem~\ref{th:main-partone}, the ergodic SRB measures of $ \Phi_{f,Q} $ enjoy the following properties: their number equals the number of the ergodic acips of $ \psi_P $, which is bounded from above by $ n $, and their supports are pairwise disjoint.
\end{corollary}

%
%

\subsection{Bernoulli components}
We now prove that for every $ \nu \in \SRB(\psi_P) $, the number of Bernoulli components of $ \Theta_{f,Q}(\nu) $ equals the number of exact components of $ \nu $. 

\begin{proposition}
\label{pr:period2}
Under the hypotheses of Theorem~\ref{th:main-partone}, if $ \nu \in \SRB(\psi_P) $, then the number of Bernoulli components of $ \Theta_{f,Q}(\nu) $ is a multiple of the number of exact components of $ \nu $.
\end{proposition}

\begin{proof}
Let $ \nu $ be an ergodic acip of $ \psi=\psi_P $, and suppose that $ \nu $ has $ m $ exact components. Accordingly, $ \psi^m $ has $ m $ exact invariant measures $ \nu_1,\ldots,\nu_m $ whose arithmetic average is equal to $ \nu $ and such that $ \psi_*(\nu_i)=\nu_{i+1} $ for each $ i = 1,\ldots,m-1 $, and $ \psi_* \nu_m = \nu_1 $. Of course, $ \nu_1,\ldots,\nu_m $ are ergodic acips of $ \psi^m $. In fact, they are the only ergodic acips of $ \psi^m $ with support contained in $ \supp \nu $.

The map $ \psi^m $ is piecewise expanding, and satisfies Condition~(*), since so does $ \psi $. Hence, Proposition~\ref{pr:boundaryacip} and Remark~\ref{re:separated} apply to the measures $ \nu_1,\ldots,\nu_m $, and so their supports are pairwise disjoint, and consist of finitely many closed intervals. Since $ \supp \nu = \bigcup^m_{i=1} \supp \nu_i $, the trapping set $ U(\nu) $ of $ \psi $ in Proposition~\ref{pr:trapping1} can be written as $ U(\nu)=\bigcup^m_{i=1} U_i $, where $ U_i $ is the subset of $ U(\nu) $ containing $ \supp \nu_i $. Since $ \psi_* $ permutes cyclically $ \nu_1,\ldots,\nu_m $, we have $ \psi(U_i) \subset U_{i+1} $ for each $ i=1,\ldots,m-1 $, and $ \psi(U_m) \subset U_1 $.

In view of the last conclusion, the trapping set $ W(\nu) $ of $ \Phi = \Phi_{f,Q} $ in Proposition~\ref{pr:trapping3} can be written as $ W(\nu) = \bigcup^m_{i=1} W_i $, where $ W_i = U_i \times (-\pi \lambda(f)/2,\pi \lambda(f)/2) $. Moreover, from the properties of $ U_1,\ldots,U_m $, it follows that $ \Phi(W_i) \subset W_{i+1} $ for each $ i=1,\ldots,m $, and $ \Phi(W_m) \subset W_1 $.

Suppose that $ \mu = \Theta_{f,Q}(\nu) $ has $ n $ Bernoulli components $ B_1,\ldots,B_n $. For every $ 1 \le i \le n $ and every $ 1 \le j \le m $, define $ B_{i,j} = B_i \cap W_j $. We claim that for each $ i $, there exists $ k $ such that $ \mu(B_{i,k}) = \mu(B_i) $. The proof of the claim is as follows. Since $ \Phi^m(W_j) \subset W_j $ for every $ j $ and $ \Phi^n(B_i) = B_i $ for every $ i $, each $ B_{i,j} $ is $ \Phi^{mn} $-forward invariant. Moreover, since $ W_1,\ldots,W_m $ are pairwise disjoint, so are $ B_{i,1}, \ldots,B_{i,m} $. But the normalization of $ \mu $ to $ B_i $ is mixing for $ \Phi^{mn} $, and so $ B_{i,1}, \ldots,B_{i,m} $ are $ \Phi^{mn} $-forward invariant and pairwise disjoint only if there exists $ k $ such that $ \mu(B_{i,k})=\mu(B_i) $.

Now, consider the set $ B_{i,k} $ such that $ \mu(B_{i,k})=\mu(B_i) $. By the invariance of $ \mu $, we have $ \mu(\Phi^n(B_{i,k}))=\mu(B_i) $, and so $ \mu(B_{i,k} \cap \Phi^n(B_{i,k})) = \mu(B_i)>0 $. Thus, there exists a nonempty set $ B \subset B_{i,k} $ such that $ \Phi^n(B) \in B_{i,k} \subset W_k $. Since $ W_k \cap \Phi^i(W_k) = \emptyset $ for all $ 1 \le i \le m-1 $, and $ \Phi^m(W_k) \subset W_k $, we can conclude that $ n $ must be a multiple of $ m $. 	
\end{proof}

\begin{theorem}
\label{th:bernoulli}
Under the hypotheses of Theorem~\ref{th:main-partone}, for every $ \nu \in \SRB(\psi_P) $, the number of Bernoulli components of $ \Theta_{f,Q}(\nu) $ equals the number of exact components of $ \nu $.
\end{theorem}

\begin{proof} 
Let $ \nu \in \SRB(\psi_P) $, and suppose that $ \nu $ has $ m $ exact components. Let $ E $ be the ergodic component of $ \Phi_{f,Q} $ corresponding to $ \mu=\Theta_{f,Q}(\nu) $. The measure $ \mu $ is the arithmetic average of $ n $ Bernoulli invariant measures of $ \Phi'=\Phi^n_{f,Q}|_{E} $. Let $ \mu' $ be one of these measures. Next, define the sets $ R' $ and $ \Delta' $ for the measure $ \mu' $ and the map $ \Phi^n $ exactly as the sets $ R_i $ and $ \Delta_i $ have been defined for the measure $ \mu_i $ and the map $ \Phi $ in Definitions~\ref{de:generic} and \ref{de:typical}. Finally, given $ x \in \Delta' $, let $ V_x $ be the open disk in $ W^u_{loc}(x) $ containing $ x $ as in Definition~\ref{de:typical}.

Remark~\ref{re:full} applies to $ \Delta' $ and $ \mu' $, and so $ \mu'(\Delta')=1 $. Let $ x \in \Delta' \cap W(\nu) $. By Lemma~\ref{le:due}, there exist an integer $ j \ge 0 $ and $ s_1 \in \FF(\nu) $ such that $ B(s_1,\eta/3) \subset \pi_s(\Phi'^j(V_x)) $, and by property~(3) of the definition of $ \FF(\nu) $, there is $ s_2 \in \FF \cap B(s_1,\eta/9) $ with $ s_2 \neq s_1 $ such that the great common divisor of $ s_1 $ and $ s_2 $ is $ m $.

Let $ x_1 $ and $ x_2 $ be the points in $ \FF_{f,Q}(\nu) $ corresponding to $ s_1 $ and $ s_2 $. Lemma~\ref{le:distance} guarantees that for each $ i=1,2 $, $ W^s_{loc}(x_i) $ intersects both lines $ \theta = - \pi \lambda(f)/2 $ and $ \theta = \pi \lambda(f)/2 $ at points $ x^-_i $ and $ x^+_i $ with $ s $-coordinate contained in $ B(s_i,2\eta/9) $. This together with $ s_2 \in B(s_1,\eta/9) $ implies that $ \pi_s(x^-_i) $ and $ \pi_{s}(x^+_i) $ are both contained in $ B(s_1,\eta/3) $, and so $ W^s_{loc}(x_i) $ intersects $ \Phi'^j(V_x) $. By Proposition~\ref{pr:period}, $ n $ is a divisor of $ m $. On the other hand, by Proposition~\ref{pr:period2}, $ n $ is a multiple of $ n $. We conclude that $ n=m $.
\end{proof}

\subsection{Basins of the ergodic SRB measures}
In~\cite[Theorem~5.1]{lxds}, we proved that for every polygon $ Q $  without parallel sides facing each other and for every reflection law $ f \in \RR^2 $ satisfying the additional condition $ f'>0 $, the ergodic SRB measures of $ \Phi_{f,Q} $ have the property that the union of their basins is a subset of $ M_Q $ of full $ \Vol $-measure. The extra hypothesis $ f'>0 $ was required to make sure that the billiard map admits SRB measures, and it can be safely replaced by the weaker condition $ f \in \RR^2 $, once we know that $ \Phi_{f,Q} $ does admit ergodic SRB measures. Accordingly, from~\cite[Theorem~5.1]{lxds}, we obtain the following.

\begin{theorem}
\label{th:basin}
Under the hypotheses of Theorem~\ref{th:main-partone}, the union of the basins of the ergodic SRB measures of $ \Phi_{f,Q} $ is a subset of $ M_Q $ of full $ \Vol $-measure.
\end{theorem}

\section{Billiards in regular polygons and triangles}
\label{sec:examples}


In this section, we apply Theorems~\ref{thm:main2} and~\ref{th:main3}
to billiards in convex regular polygons with an odd number of sides and to billiards in triangles. The case of the billiard in an equilateral triangle was first studied in \cite{arroyo12}.

\begin{lemma}
\label{lem:regpoly}
Each convex regular polygon with an odd number $ n $ of sides satisfies Condition~(*).
\end{lemma}

\begin{proof}
Let $ P $ be a convex regular polygon with an odd number $ n $ of sides. By~\cite[Lemma 3.1]{MDDGP13-1} the slap map $\psi_{P}$ is conjugated to the skew-product map $F_n \colon [0,1]\times\mathbb{Z}_n\to[0,1]\times\mathbb{Z}_n$ defined by 
\[
F_{n}(x,y)=\left(\phi_{n}(x),\sigma_x(y)\right)
\]
with
\[
\phi_{n}(x)=-\frac{1}{\cos(\frac{\pi}{n})} \left(x-\frac{1}{2}\right) \pmod 1,
\]
and
\[
\sigma_x(y)=
\begin{cases}
y - \left[\frac{n}{2}\right] & \text{if } x<\frac{1}{2} \\
y + \left[\frac{n}{2}\right] & \text{if } x>\frac{1}{2}
\end{cases}.
\]
%
It was proved in~\cite[Proposition 6.2 and Corollary 6.3]{MDDGP13} that the discontinuity point of $\phi_n$ is not pre-periodic. Since $\phi_n$ has a single discontinuity, $P$ satisfies Condition~(*). 
\end{proof}


%

\begin{corollary}
\label{co:regpoly}
Let $ P $ be a convex regular polygon with an odd number $n\geq3$ of sides. There exists $ \delta>0 $ such that if $ f\in \RR^2 $ with $ \lambda(f)<\delta $ and $ Q \in \PP_n $ with $ d(Q,P)<\delta $, then the following hold:
\begin{enumerate}
\item if $n=3$, then $\Phi_{f,Q}$ has a unique ergodic SRB measure, and this measure has a single Bernoulli component;
\item if $n=5$, then $\Phi_{f,Q}$ has a unique ergodic SRB measure, and this measure has two Bernoulli components; 
\item if $n\geq7$, then $\Phi_{f,Q}$ has exactly $n$ ergodic SRB measures. All these measures have $2^{m(n)}$ Bernoulli components, where $m(n)$ is the integer part of $ -\log_2(-\log_2 \cos(\pi/n)) $.
\end{enumerate}
\end{corollary}

\begin{proof}
By \cite[Theorem 1.1]{MDDGP13-1}, the slap map $ \psi_{P} $ has a unique ergodic acip if $ n=3 $ or $ n=5 $, and exactly $ n $ ergodic acips if $ n \ge 7 $. Moreover, every acip has $2^{m(n)}$ exact components, where $m(n)$ is the integer part of $ -\log_2(-\log_2 \cos(\pi/ n)) $. In particular, $m(3)=0$ and $m(5)=1$. The conclusion of the corollary now follows from Theorem~\ref{thm:main2} and Lemma~\ref{lem:regpoly}.
\end{proof}


Next, we consider billiards in triangles.

\begin{corollary}
\label{co:triangle2}
There exists a residual and full measure subset $\XX_3$ of the set of triangles $\PP_3$ with the following property: for any $P \in \XX_3$, there is $ \delta>0 $ such that if $ f\in \RR^2 $ with $ \lambda(f)<\delta $ and $ Q \in \PP_3 $ with $ d(Q,P)<\delta $, then the  billiard map $\Phi_{f,Q}$ has a unique ergodic SRB measure. This measure is Bernoulli if $P$ is acute, and has an even number of Bernoulli components, otherwise.
\end{corollary}

\begin{proof}
The corollary follows from Theorem~\ref{th:main3} and the ergodic properties of slap maps of triangles~\cite[Theorem 1.2]{MDDGP13-1}.
\end{proof}


\begin{corollary}
\label{co:triangle1}
Let $ P $ be an acute triangle. If $ \lambda \in \RR^2 $ and $ \lambda(f) $ is sufficiently small, then $\Phi_{f,P}$ has a unique ergodic SRB measure, and this measure has a single Bernoulli component.  
\end{corollary}

\begin{proof}
By \cite[Theorem 1.2]{MDDGP13-1}, the slap map of any acute triangle has a unique ergodic acip, which is also exact. Since acute triangles trivially satisfy Condition~(*), the wanted conclusion follows from Theorem~\ref{thm:main2}.
\end{proof}


Based on the previous two corollaries, we formulate the conjecture:

\begin{conjecture}
Let $ P $ be an obtuse triangle. If $ f \in \RR^2 $ and $ \lambda(f) $ is sufficiently small, then $\Phi_{f,P}$ has a unique ergodic SRB measure, and this measure has an even number of Bernoulli components.
\end{conjecture}

\section*{Acknowledgements}

The authors were partially funded by
the Project `New trends in Lyapunov exponents' (PTDC/MAT-PUR/29126/2017).
The authors JLD and JPG were partially supported by the Project CEMAPRE - UID/MULTI/00491/2019 financed by FCT/MCTES through national funds and PD through the strategic project PEst-OE/MAT/UI0209/2013.
The authors wish to thank an anonymous referee for reading carefully a previous version of the manuscript and making many useful remarks and suggestions.


\end{document}